\documentclass[12pt]{amsart}
\usepackage{epsf,overpic,amssymb,amscd,times,stmaryrd,euscript,times,psfrag}
\usepackage{tikz-cd}
\usepackage{hyperref}

\newtheorem{thm}{Theorem}[section]
\newtheorem*{thm*}{Theorem}

\newtheorem{prop}[thm]{Proposition}

\newtheorem{defn}[thm]{Definition}

\newtheorem{lemma}[thm]{Lemma}
\newtheorem{cor}[thm]{Corollary}

\newtheorem{q}[thm]{Question}
\newtheorem{rmk}[thm]{Remark}

\newtheorem*{hypH}{Hypothesis H}

\numberwithin{equation}{subsection}

\newcommand{\C}{\mathbb{C}}
\newcommand{\R}{\mathbb{R}}
\newcommand{\Z}{\mathbb{Z}}
\newcommand{\Q}{\mathbb{Q}}
\newcommand{\N}{\mathbb{N}}

\newcommand{\F}{\mathcal{F}}

\newcommand{\bdry}{\partial}
\newcommand{\s}{\vskip.1in}
\newcommand{\n}{\noindent}

\newcommand{\op}{\operatorname}

\newcommand{\bs}{\boldsymbol}

\makeatletter
\newcommand*\bigcdot{\mathpalette\bigcdot@{.5}}
\newcommand*\bigcdot@[2]{\mathbin{\vcenter{\hbox{\scalebox{#2}{$\m@th#1\bullet$}}}}}
\makeatother

\newcommand{\be}{\begin{enumerate}}
\newcommand{\ee}{\end{enumerate}}

\begin{document}

\title[Topological entropy for Reeb vector fields]{Topological entropy for Reeb vector fields in dimension three via open book decompositions}

\author{Marcelo R.R. Alves}
\address{Institut de Math\'ematiques,
Universit\'e de Neuch\^atel,
Rue \'Emile Argand 11,
CP~158,
2000 Neuch\^atel,
Switzerland}
\email{marcelo.ribeiro@unine.ch} \urladdr{}

\author{Vincent Colin}
\address{Universit\'e de Nantes, UMR 6629 du CNRS, 44322 Nantes, France}
\email{vincent.colin@univ-nantes.fr}

\author{Ko Honda}
\address{University of California, Los Angeles, Los Angeles, CA 90095}
\email{honda@math.ucla.edu} \urladdr{http://www.math.ucla.edu/\char126 honda}

\date{}



\thanks{MA supported by the Swiss National Foundation. VC supported by the ERC G\'eodycon and ANR Quantact. KH supported by NSF Grant DMS-1406564.  }

\begin{abstract}

Given an open book decomposition of a contact three manifold $(M,\xi )$ with pseudo-Anosov monodromy and fractional Dehn twist coefficient $c={k\over n}$,
we construct a Legendrian knot $\Lambda$ close to the stable foliation of a page, together with a small Legendrian pushoff $\widehat{\Lambda}$. When $k\geq 5$,
we apply the techniques of \cite{CH2} to show that the strip Legendrian contact homology of $\Lambda \rightarrow\widehat{\Lambda}$  is well-defined and has an exponential growth property. The work \cite{Al} then implies that all Reeb vector fields for $\xi$ have positive topological entropy.
\end{abstract}

\maketitle

\tableofcontents

\section{Introduction}

In this paper we combine the techniques of \cite{CH2} and \cite{Al} to obtain results on the topological entropy of a large class of contact $3$-manifolds.

\begin{thm}\label{thm: main}
Let $(M,\xi )$ be a closed cooriented three-dimensional contact manifold which admits a supporting open book decomposition whose binding is connected and whose monodromy is isotopic to a pseudo-Anosov homeomorphism with fractional Dehn twist coefficient $\frac{k}{n}$. If $k \geq 5$ then every Reeb vector field for $\xi$ has positive topological entropy.
\end{thm}

In order to explain the significance of our result we start by recalling some notions. A \textit{cooriented contact structure} on a $3$-manifold $M$ is a plane field $\xi$  given by $\xi = \ker \alpha$ for a $1$-form $\alpha$ with $\alpha \wedge d\alpha > 0$. Such an $\alpha$ is called a \textit{contact form} on $(M,\xi)$, and it determines a \textit{Reeb vector field} $R_{\alpha}$ defined by $\iota_{R_{\alpha}}d\alpha = 0$, $\alpha(R_{\alpha})= 1$. Denote the \textit{Reeb flow of $R_\alpha$} by $\phi_{\alpha}= (\phi_{\alpha}^t)_{t\in \R}$. A \textit{Legendrian} submanifold of $(M,\xi)$ is a $1$-dimensional submanifold of $M$ everywhere tangent to $\xi$.  {\em In what follows we assume that our contact structures are cooriented.}

The {\em topological entropy $h_{top}$} is a nonnegative number that one associates to a dynamical system and which measures its complexity.  We briefly review its definition from \cite{Bo} for a flow $\phi=(\phi_t)_{t\in \R}$. Fix a metric $d$ on $M$.  Given $T>0$ and $x,y\in M$, define
$$d_T(x,y)=\max_{t\in[0,T]} d(\phi^t(x),\phi^t(y)).$$
A subset $X\subset M$ is {\em $(T,\epsilon)$-separated} if for all $x,y\in X$ we have $d_T(x,y)>\epsilon$. Writing $N(T,\epsilon)$ for the maximum cardinality of a $(T,\epsilon)$-separated subset of $M$, we define
$$h_{top}(\phi)=\lim_{\epsilon\to 0}\limsup_{T\to +\infty} \frac{\log N(T,\epsilon)}{T}.$$

The positivity of the topological entropy for a dynamical system implies some type of exponential instability for that system. For $3$-dimensional flows, the positivity of $h_{top}$ has the following striking dynamical consequence due to Katok \cite{Katok,K2} and Lima and Sarig \cite{LS,S}:

\begin{thm*}
Let $\phi$ be a smooth flow on a closed oriented $3$-manifold generated by a nonvanishing vector field. If $h_{top}(\phi)>0$, then there exists a Smale ``horseshoe'' as a subsystem of the flow. As a consequence, the number of hyperbolic periodic orbits of $\phi$ grows exponentially with respect to the period.
\end{thm*}

A ``horseshoe'' is a compact invariant set where the dynamics is semi-conjugate to that of the suspension of a finite shift by a finite-to-one map and is considered to be the prototypical example of chaotic dynamics; see \cite{KH}.
Combining Theorem \ref{thm: main} with this result we obtain the following corollary which is a strengthening of a result proved in \cite{CH2}.

\begin{cor}
Let $(M,\xi )$ be a closed contact $3$-manifold which admits a supporting open book decomposition whose binding is connected and whose monodromy is isotopic to a pseudo-Anosov homeomorphism with fractional Dehn twist coefficient $\frac{k}{n}$. If $k\geq 5$, then for every Reeb flow $\phi_\alpha$ on $(M,\xi)$ there exists a ``horseshoe" as a subsystem of $\phi_\alpha$.  In particular the number of hyperbolic periodic orbits of $\phi_\alpha$ grows exponentially with respect to the period, even if $\alpha$ is not generic.
\end{cor}

Motivated by results on topological entropy for geodesic flows, Macarini and Schlenk used the geometric ideas of \cite{FS1,FS2} to prove in \cite{MS} that if $Q$ is a surface with genus $\geq 2$, then on the unit cotangent bundle $(S^{*}Q, \xi)$ equipped with the canonical contact structure $\xi$ every Reeb flow has positive topological entropy. In previous works \cite{A,Al,ReebAnosov}, the first author discovered an abundance of examples of contact $3$-manifolds with positive entropy: there exist hyperbolic contact $3$-manifolds, nonfillable contact $3$-manifolds, as well as $3$-manifolds with infinitely many nondiffeomorphic contact structures for which every Reeb flow has positive topological entropy.

Since every contact $3$-manifold admits a supporting open book decomposition with pseudo-Anosov monodromy, connected binding, and $k\geq 1$ by \cite{CH1}, Theorem \ref{thm: main} can be interpreted as saying that for ``almost all'' (apart from $k=1,\dots,4$) tight contact $3$-manifolds every Reeb flow has positive topological entropy.

Recall that the second and third authors had previously shown in \cite{CH2} that, when $k\geq 3$, for every nondegenerate Reeb vector field the number of periodic orbits grows exponentially with the action and in the degenerate case the total number of periodic orbits is infinite. The proof uses contact homology as follows:
\begin{enumerate}
\item Construct a contact form $\alpha$ on $M$ supported by the open book decomposition, whose Reeb vector field $R=R_\alpha$ behaves well with respect to the pseudo-Anosov monodromy.
\item Prove that, when $k\geq 2$, the symplectization $(\R\times M, J)$ of $(M,\alpha)$ contains no $J$-holomorphic plane asymptotic to a periodic orbit of $R$ at $+\infty$.
\item Prove that, when $k\geq 3$, index one $J$-holomorphic cylinders between periodic orbits of $R$ cannot intersect the trivial cylinder over the binding.
\end{enumerate}

Step (2) implies that the cylindrical contact homology of $(M,\alpha)$ is well-defined. By Step (3), a $J$-holomorphic cylinder must join two orbits in the same Nielsen class. The key point in the proofs of Steps (2) and (3) is that every $J$-holomorphic curve in $\R\times M$ must intersect the $J$-holomorphic trivial cylinder over the binding positively. By using a suitable Rademacher function $\Phi$ (cf.\ Section~\ref{subsection: Rademacher function}), one can show that the intersection must be empty as soon as $k$ is sufficiently large. One finally uses properties of diffeomorphisms isotopic to pseudo-Anosov homeomorphisms: in every Nielsen class of periodic points, the total Lefschetz index is $-1$ and the number of Nielsen classes grows exponentially with the period. This is enough to conclude that the dimension of cylindrical contact homology generated by periodic orbits of action less than $T$ grows exponentially with $T$. Invariance properties of the set of linearized contact homologies yield the result for arbitrary contact forms.

To prove Theorem \ref{thm: main} we follow the same strategy combined with the following criterion which generalizes \cite[Theorem 1]{Al}. 

\begin{thm}\label{thm: criterion}
Let $(M,\xi =\ker \alpha_0)$ be a closed contact manifold and $\Lambda$ and $\widehat{\Lambda}$ be a pair of disjoint Legendrian knots in $(M,\xi)$. If there exists an exhaustive sequence $(\alpha_i=G_i\alpha_0,L_i)$ of contact forms and actions such that
\be
\item $G_i$ is uniformly bounded above by a constant ${\frak c}\geq 1$ and below by ${1\over {\frak c}}$;
\item the strip Legendrian contact homologies $LCH_{st}^{\leq L_i}(\alpha_i, \Lambda \rightarrow\widehat{\Lambda})$ are well-defined and grow exponentially with the action in the following sense: there exist numbers $a>0$ and $b$ such that
$$\dim(LCH_{st}^{\leq \ell}(\alpha_i,\Lambda \rightarrow\widehat{\Lambda})\geq e^{a\ell +b}$$
for each $i\in\N$ and $\ell\leq L_i$; and
\item for each $i\leq j$ there are induced cobordism maps
$$\Psi_{\mathcal{W}^i_j}: LCH_{st}^{\leq L_i}(\alpha_i, \Lambda \rightarrow\widehat{\Lambda})\to LCH_{st}^{\leq L_j}(\alpha_j, \Lambda \rightarrow\widehat{\Lambda}),$$
that are injective on $LCH_{st}^{\leq L_i/{\frak d}}(\alpha_i, \Lambda \rightarrow\widehat{\Lambda})$, where ${\frak d}\geq 1$ is independent of $i,j$,
\ee
then every Reeb flow on $(M,\xi)$ has positive topological entropy.
\end{thm}

We refer the reader to Section~\ref{section: review of contact homology} for the definition of strip Legendrian contact homology and Definition~\ref{defn: exhaustive} for the definition of an exhaustive sequence.

By Theorem~\ref{thm: criterion} it suffices to find a pair of suitable Legendrian knots $\Lambda$ and $\widehat{\Lambda}$ that satisfy the conditions of Theorem~\ref{thm: criterion}. The knot $\Lambda$ will be constructed in a neighborhood of a page of the open book as the Legendrian lift of an immersed curve $L$ in the page which is close to the stable invariant foliation of the pseudo-Anosov representative of the monodromy. It will have the property that the value of the Rademacher function $\Phi$ is $0$ on every arc immersed in $L$. The knot $\widehat{\Lambda}$ is a small pushoff of $\Lambda$ in the Reeb direction.  The control on $J$-holomorphic curves required in Theorem~\ref{thm: criterion} is then obtained using the method from Steps (2) and (3) of the absolute case, provided $k$ is large enough to take care of the extra leaking coming from the parts in the boundary of potential $J$-holomorphic curves lying in $\R\times\Lambda$ and $\R\times \widehat{\Lambda}$.

\begin{cor}
Every closed contact three-dimensional manifold $(M,\xi)$ admits a (tight) degree five branched cover along a transverse knot on which every Reeb flow has positive topological entropy.
\end{cor}

\begin{proof}
Using \cite{CH1} we take a supporting open book for $(M,\xi)$ with pseudo-Anosov monodromy $h$, connected binding $K$, and fractional Dehn twist coefficient $\frac{k}{n}$ with $k\geq 1$. The degree five branched cover of $(M,\xi)$ along $K$ is supported by an open book decomposition whose monodromy is $h^5$.  Its fractional Dehn twist coefficient is $\frac{5k}{n}$ with $5k\geq 5$ and we can apply Theorem~\ref{thm: main}.
\end{proof}

We finish the introduction with a question which we believe could be interesting for future investigations:

\begin{q}
Does there exist an overtwisted contact $3$-manifold on which every Reeb flow has positive topological entropy?
\end{q}

\textbf{Acknowledgements:} Our special thanks to Fr\'ed\'eric Bourgeois for explaining to us the construction of the linearized Legendrian contact homology of a directed pair of Legendrian submanifolds, which is part of his joint work with Ekholm and Eliashberg \cite{BEE}. 

\section{Preliminaries}

\subsection{Open book decompositions}

\begin{defn}
An {\em open book decomposition of} a closed $3$-manifold $M$ is a triple $(S,h,\phi)$, where
\begin{itemize}
\item $S$ is a compact oriented surface with nonempty boundary;
\item the {\em monodromy} $h:S\stackrel\sim\to S$ is a diffeomorphism which restricts to the identity on $\partial S$; and
\item $\phi : M(S,h) \stackrel\sim\to M$ is a homeomorphism.
\end{itemize}
Here $M(S,h)$ denotes the {\em relative mapping torus} $S\times [0,1]/\sim_h$, where $\sim_h$ is the equivalence relation given by $(h(x),0)\sim_h (x,1)$ for all $x\in S$ and $(y,t)\sim_h (y,t')$ for all $y\in \partial S$ and $t,t' \in [0,1]$.
\end{defn}

The surfaces $S\times \{ t\}/\sim_h$ are the {\em pages} of the open book and the link $\partial S\times [0,1]/\sim_h$ is the {\em binding}. (We will also refer to their images under $\phi$ as the ``pages'' and the ``binding''.) Note that the manifold $M(S,h)$ is oriented since $S$ and $[0,1]$ are and $M$ has the induced orientation via $\phi$.  The binding is oriented as the boundary of $S$.

The relationship between contact structures and open book decompositions was clarified by Giroux \cite{Gi}, who notably gave the following definition:

\begin{defn}
An open book decomposition $(S,h)$ for $M$ {\em supports} a contact structure $\xi$ if
\begin{itemize}
\item $\xi$ is positive with respect to the orientation of $M(S,h)$; and
\item there exists a contact form $\alpha$ for $\xi$ whose Reeb vector field is positively transverse to the interior of the pages and positively tangent to the binding.
\end{itemize}
\end{defn}

A fundamental result of Giroux~\cite{Gi} is that every contact structure on a closed $3$-manifold is supported by some open book decomposition.

\subsection{Fractional Dehn twist coefficients}

Let $S$ be a compact oriented surface with nonempty connected boundary and $h: S\stackrel\sim\to S$ be a diffeomorphism of $S$ which is the identity on the boundary $\partial S$. We take an orientation-preserving identification $\partial S \simeq\R/\Z$. Suppose that $h$ is freely homotopic to a pseudo-Anosov homeomorphism $\psi$ by a homotopy $(h_t )_{t\in [0,1]}$, $h_0 =h$, $h_1 =\psi$.

We define the {\em fractional Dehn twist coefficient $c$ of h} to be the rotation number of the isotopy $(\beta_t =h_t \vert_{\partial S} )_{t\in [0,1]}$ as follows: Lift $(\beta_t  )_{t\in [0,1]}$ from an isotopy of $\R /\Z$ to an isotopy $(\widetilde{\beta_t} )_{t\in [0,1]}$ of $\R$ and define $f(x) =\widetilde{\beta_1} (x) -\widetilde{\beta_0} (x)+x$. Then
$$c=\lim_{n\rightarrow \infty} {f^n(x)-x\over n}$$
for any $x\in \bdry S$.

The pseudo-Anosov map $\psi$ has a stable invariant singular foliation $\F$. Since $\partial S\simeq \R/\Z$, the foliation $\F$ has a certain number $n$ of singularities $x_1,\dots,x_n$ along $\partial S$. These singularities are preserved by both $h$ (which is the identity on $\partial S$) and $\psi$. This implies that the fractional Dehn twist coefficient of $h$ has the form $c=\frac{k}{n}$.

\subsection{The Rademacher function for a pseudo-Anosov map} \label{subsection: Rademacher function}

Let $\psi : S\stackrel\sim\to S$ be a pseudo-Anosov homeomorphism of a compact oriented surface with nonempty connected boundary. We define a quasi-morphism $\Phi: \pi_1(S) \to \Z$  that is invariant under $\psi$, called the {\em Rademacher function}. Our actual map will rather be defined on the space of homotopy classes of oriented arcs in $S$ relative to their endpoints. It has the following properties:
\be
\item $\Phi (\delta_0 \delta_1) =\Phi (\delta_0 )+\Phi (\delta_1 )+\epsilon$, $\epsilon =-1,0,1$;
\item $\Phi (e)=0$; and
\item $\Phi (\delta )=-\Phi (-\delta)$.
\ee
Here $\delta_0$ and $\delta_1$ are arcs in $S$ where the terminal point of $\delta_0$ equals the initial point of $\delta_1$ and $\delta_0 \delta_1$ is their concatenation; $e$ is any constant arc; and $-\delta$ is $\delta$ with reversed orientation.

To construct $\Phi$, we consider the stable invariant foliation $\F$ of $\psi$ with saddle singularities $x_1,\dots,x_n$ on $\partial S$. We denote by $P_i$, $i=1,\dots, n$, the separatrix in $int (S)$ that limits to $x_i$; it is called a {\em prong}.  By slight abuse of notation we assume that $P_i$ contains $x_i$. We orient the prongs so that their intersection point with $\partial S$ is positive. Let $\delta$ be an oriented arc in $S$ and $\widetilde{\delta}$ a lift of $\delta$ to the universal cover $\widetilde{S}$ of $S$. For any component $\widetilde{D}$ of $\partial \widetilde{S}$, we consider $P_{\widetilde{D}}$, the union of all the lifts of $P_1,\dots,P_n$ with initial point on $\widetilde{D}$. If the (algebraic) intersection number of $\widetilde{\delta}$ with $P_{\widetilde{D}}$ is greater than or equal to $2$, we set $n_{P_{\widetilde{D}}} (\widetilde{\delta} )$ to be the intersection number of $\widetilde{\delta}$ and $P_{\widetilde{D}}$ minus $1$. If the intersection number of $\widetilde{\delta}$ with $P_{\widetilde{D}}$ is less than or equal to $-2$, we set $n_{P_{\widetilde{D}}} (\widetilde{\delta} )$ to be the intersection number of $\widetilde{\delta}$ and $P_{\widetilde{D}}$ plus $1$. Otherwise, we set $n_{P_{\widetilde{D}}} (\widetilde{\delta} )=0$. Finally we define
\begin{equation}
\Phi (\delta )=\Sigma_{\widetilde{D} \in \partial \widetilde{S}}\; n_{P_{\widetilde{D}}} (\widetilde{\delta} ).
\end{equation}

The properties of $\Phi$ are proven in \cite{CH2}. In particular, the sum has finitely many nonzero terms.

\section{Review of contact homology} \label{section: review of contact homology}

In this section we review the linearized Legendrian contact homology of a directed pair of Legendrian submanifolds $\Lambda \rightarrow\widehat{\Lambda}$ in a contact manifold $(M,\xi)$. This was explained to us by Fr\'ed\'eric Bourgeois and is part of his joint work \cite{BEE} with Ekholm and Eliashberg.\footnote{Any mistakes or imprecise statements in the presentation are our responsibility.}

A contact form $\alpha$ for $(M,\xi)$ is {\em nondegenerate} if for each Reeb orbit of $\alpha$ the linearized return map does not have $1$ as an eigenvalue.  Given a Legendrian submanifold $\Lambda$, $\alpha$ is {\em $\Lambda$-nondegenerate} if for each Reeb chord $c: [0,T]\to M$, $(\phi_T)_*T_{c(0)}\Lambda$ is transverse to $T_{c(T)}\Lambda$, where $\phi_T$ is the time-$T$ flow of the Reeb vector field for $\alpha$.

\subsection{Cobordisms and almost complex structures}

\subsubsection{Symplectizations}

Let $(M,\xi)$ be a contact manifold and $\alpha$ a contact form for $(M,\xi)$. The symplectization of $(M,\xi)$ is the product $\mathbb{R} \times M$ with the symplectic form $d(e^s \alpha)$, where $s$ denotes the $\mathbb{R}$-coordinate on $\mathbb{R} \times M$. The $2$-form $d\alpha$ restricts to a symplectic form on the vector bundle $\xi$ and it is well-known that the set $\mathfrak{j}(\alpha)$ of $d\alpha$-compatible almost complex structures on the symplectic vector bundle $\xi$ is nonempty and contractible. Notice that if $M$ is $3$-dimensional, the set $\mathfrak{j}(\alpha)$ does not depend on the contact form $\alpha$ on $(M,\xi)$.

Given $j \in \mathfrak{j}(\alpha)$, there exists an $\mathbb{R}$-invariant almost complex structure $J$ on $\mathbb{R} \times M$ such that:
\begin{equation} \label{eq16}
J \tfrac{\bdry}{\bdry s} = R_\alpha, \ \ J|_\xi = j,
\end{equation}
where $R_\alpha$ is the Reeb vector field of $\alpha$.
We will denote by $\mathcal{J}(\alpha)$ the set of almost complex structures in $\mathbb{R} \times M$ that are $\mathbb{R}$-invariant
and satisfy \eqref{eq16} for some $j \in \mathfrak{j}(\alpha)$.

\subsubsection{Exact symplectic cobordisms}

Let $(\widehat{W} ,\varpi = d\kappa )$ be an exact symplectic manifold without boundary, and let $(M^+,\xi^+)$ and $(M^-,\xi^-)$ be contact manifolds with contact forms $\alpha^+$ and $\alpha^-$.

\begin{defn}
$(\widehat{W} ,\varpi = d\kappa )$ is an {\em exact symplectic cobordism from $\alpha^+$ to $\alpha^-$} if there exist codimension $0$ submanifolds $W^-$, $W^+$, and $W$ of $\widehat{W}$ and diffeomorphisms $\Psi^+: W^+ \stackrel\sim\to [s_+,+\infty) \times M^+$ and $\Psi^-: W^- \stackrel\sim\to (-\infty,s_-] \times M^-$ for some $s_+$, $s_-$ such that:
\begin{gather}
W \mbox{ is compact, } \widehat{W}= W^+ \cup W \cup W^-, ~~ W^+ \cap W^- = \emptyset,\\
\nonumber (\Psi^+)^* (e^s\alpha^+) = \kappa \mbox{ and } (\Psi^-)^* (e^s\alpha^-) = \kappa.
\end{gather}
\end{defn}

We will write $ \alpha^+ \succ_{ex} \alpha^-$ if there exists an exact symplectic cobordism from $ \alpha^+$ to $ \alpha^-$. We remind the reader that $\alpha^+ \succ_{ex} \alpha$ and $ \alpha \succ_{ex} \alpha^-$ imply $ \alpha^+ \succ_{ex} \alpha^-$; or in other words the exact symplectic cobordism relation is transitive; see \cite{CPT} for a detailed discussion on symplectic cobordisms with cylindrical ends. Notice that a symplectization is a particular case of an exact symplectic cobordism.

\begin{defn}
An almost complex structure $J$ on $(\widehat{W},\varpi)$ is {\em cylindrical} if
\begin{itemize}
\item $J$ coincides with $J^+ \in \mathcal{J}(\alpha^+)$ on $W^+$,
\item $J$ coincides with $J^- \in \mathcal{J}(\alpha^-)$ on $W^-$, and
\item $J$ is compatible with $\varpi$ on $\widehat{W}$.
\end{itemize}
\end{defn}

For fixed $J^+ \in \mathcal{J}(\alpha^+)$ and $J^- \in \mathcal{J}(\alpha^-)$, we denote by $\mathcal{J}(J^-,J^+)$ the set of cylindrical almost complex structures on $( \widehat{W},\varpi)$ that coincide with $J^+$ on $W^+$ and $J^-$ on $W^-$.  It is well known that $\mathcal{J}(J^-,J^+)$ is nonempty and contractible.

\subsubsection{SFT-admissible exact Lagrangian cobordisms}

Let $(\widehat W,\varpi=d\kappa)$ be an exact symplectic cobordism from $\alpha^+$ to $\alpha^-$.

\begin{defn}
A Lagrangian submanifold $L$ of $(\widehat W,d\kappa)$ is a {\em Lagrangian cobordism from $\overline{\Lambda}^+$ to $\overline{\Lambda}^-$} if there exist Legendrian submanifolds $\overline{\Lambda}^+$ of $(M^+,\ker \alpha^+)$ and $\overline{\Lambda}^-$ of $(M^-,\ker \alpha^-)$, and $N>0$ such that, with respect to the identifications $\Psi^+$ and $\Psi^-$:
\begin{gather}
L \cap ([N,+\infty) \times M^+) = ([N,+\infty) \times \overline{\Lambda}^+), \\
L \cap ((-\infty,-N] \times M^-)= ((-\infty,-N] \times \overline{\Lambda}^-).
\end{gather}
\end{defn}

If such an $L$ is an exact Lagrangian submanifold of $(\widehat W,d\kappa)$, we call it an {\em exact Lagrangian cobordism from $\overline{\Lambda}^+$ to $\overline{\Lambda}^-$.}
If $L$ is such that there exists a primitive of $\kappa |_L$  which vanishes outside some compact subset of $L$, then $L$ is called an \textit{SFT-admissible exact Lagrangian cobordism}. All exact Lagrangian cobordisms used in this paper are SFT-admissible.

\subsection{Pseudoholomorphic curves}

Let $\alpha$ be a contact form for $(M,\xi)$, $J \in \mathcal{J}(\alpha)$, and $\widetilde{\Lambda}$ be a Legendrian submanifold of $(M,\xi)$. Let $(S,i)$ be a connected compact Riemann surface, possibly with boundary, $\Omega \subset S$ be a finite ordered set, and $\Omega_{\partial} := \Omega \cap \partial S $.

\begin{defn} \label{defn: holo}
An $(i,J)$-holomorphic map $\widetilde{u}= (r,u):(S \setminus \Omega,i) \to (\mathbb{R} \times M,J)$ with boundary in $\mathbb{R}\times \widetilde{\Lambda}$ is a map that satisfies
$$\overline{\partial}_J(\widetilde{u}):= d\widetilde{u} \circ i - J \circ d\widetilde{u}=0,$$
$$\widetilde{u}(\partial S \setminus \Omega_{\partial}) \subset \mathbb{R}\times \widetilde{\Lambda}.$$
\end{defn}

The operator $\overline{\partial}_J$ above is called the {\em Cauchy-Riemann operator} for the almost complex structure $J$.

An analogous definition exists for exact symplectic cobordisms. Let $(W,d\varpi)$ be an exact symplectic cobordism from a contact form $\alpha^+$ for $(M^+,\xi^+)$ to a contact form $\alpha^-$ for $(M^-,\xi^-)$ in the sense of \cite{CPT}, and $L$ be an SFT-admissible exact Lagrangian cobordism in $(W,d\varpi)$ from a Legendrian submanifold $\widetilde{\Lambda}^+$ of $(M^+,\xi^+)$ to a Legendrian submanifold $\widetilde{\Lambda}^-$ of $(M^-,\xi^-)$.
We consider an almost complex structure $J \in \mathcal{J}(J^+,J^-)$, where $J^+ \in  \mathcal{J}(\alpha^+)$ and $J^- \in  \mathcal{J}(\alpha^-)$ and consider $(i,J)$-holomorphic maps $\widetilde{u}: (S\setminus \Omega,i)\to (\widehat {W},J)$ with boundary on $L$ in a manner analogous to Definition~\ref{defn: holo}.

\subsubsection{Hofer energy}

All pseudoholomorphic curves considered in this paper have finite Hofer energy (cf.\ \cite{H,CPT}). This is essential to guarantee that the curves behave well near the points of $\Omega$ and to ensure the compactness of the relevant moduli spaces.

By the finiteness of Hofer energy, the holomorphic maps $\widetilde{u}: (S\setminus \Omega,i)\to (\widehat{W},J)$ have a very controlled behavior near the points of $\Omega$: they are either removable or are asymptotic to Reeb orbits or Reeb chords, as shown in \cite{Ab,H,HWZ}. For simplicity we assume there are no removable points.
The elements of the set $\Omega \subset S$ are called {\em punctures} of $\widetilde{u}$; the elements of $\Omega_\partial$ are {\em boundary punctures} and the elements in $\Omega \setminus \Omega_\partial$ are {\em interior punctures.} The works \cite{Ab,H,HWZ} classify the punctures in four different types (here we are writing $\widetilde{u}=(r,u)$ at the ends):
\begin{itemize}
\item $z \in \Omega$ is a positive (resp.\ negative) boundary puncture if $z \in \Omega_\partial$ and $\lim_{z' \to z} s(z') = +\infty$ (resp.\ $-\infty$); in this case $u$ is positively (resp.\ negatively) asymptotic to a Reeb chord near $z$;
\item $z \in \Omega$ is a positive (resp.\ negative) interior puncture if $z \in \Omega \setminus \Omega_\partial$ and $\lim_{z' \to z} s(z') = +\infty$ (resp.\ $-\infty$); in this case $u$ is positively (resp.\ negatively) asymptotic to a Reeb orbit near $z$.
\end{itemize}
Intuitively, on a neighborhood of a puncture, $\widetilde{u}$ detects a Reeb chord or a Reeb orbit. If $u$ is asymptotic to a $\Lambda$-nondegenerate Reeb chord or a nondegenerate Reeb orbit at a puncture, more can be said about its asymptotic behavior in the neighborhood of this puncture. In this case $u$ limits to Reeb chords and Reeb orbits exponentially fast at the punctures: this is crucial to define a Fredholm theory for $J$-holomorphic curves.

We will now define the three different types of pseudoholomorphic curves that are used in this paper.

\subsubsection{Curves of \textbf{Type A}} \label{subsubsection: A}

Let $\alpha$ be a contact form for $(M,\xi)$ and $J\in \mathcal{J}(\alpha)$. Let $\gamma^+$ be a nondegenerate Reeb orbit of $\alpha$ and let $\bs\gamma^-=(\gamma^-_1,\gamma^-_2,\dots,\gamma^-_n)$ be an ordered sequence of nondegenerate Reeb orbits of $\alpha$.

\begin{defn}
A finite energy $J$-holomorphic curve of {\em \textbf{Type A} in $(\mathbb{R}\times M,J)$ asymptotic to $(\gamma^+,\bs\gamma^-)$}  is a tuple $(\widetilde{u},z^+,{\bf z}^-)$, where $z^+$ is a point in $S^2$, ${\bf z}^-=(z^-_1,\dots,z^-_n)$ is an ordered collection of distinct points $\not=z^+$ on $S^2$, and $\widetilde{u}: (S^2\setminus z^+ \cup {\bf z}^-,i_0) \to  (\mathbb{R}\times M,J)$ is a finite energy $J$-holomorphic curve satisfying:
\begin{itemize}
  \item $\widetilde{u}$ is positively asymptotic to $\gamma^+$ near $z^+$; and
  \item $\widetilde{u}$ is negatively asymptotic to $\gamma^-_l$ near $z^-_l$, $l=1,\dots, n$.
\end{itemize}
Here $i_0$ is the standard complex structure on $S^2$.
\end{defn}

Let $\widetilde{\mathcal{M}}(\gamma^+,\bs\gamma^-;J)$ be the moduli space of equivalence classes of finite energy $J$-holomorphic curves of \textbf{Type A} in $(\mathbb{R}\times M,J)$ asymptotic to $(\gamma^+,\bs\gamma^-)$. Two $J$-holomorphic curves $(\widetilde{u},z^+,{\bf z}^-)$ and $(\widetilde{v},y^+,{\bf y}^-)$ of \textbf{Type A} in $(\mathbb{R}\times M,J)$ asymptotic to $(\gamma^+,\bs\gamma^-)$ are equivalent if there exists a biholomorphism $\tau:(S^2,i_0)\stackrel\sim \to (S^2,i_0)$ which satisfies:
\begin{itemize}
   \item $\tau(z^+)=y^+ $ and $\tau(z^-_l)= y^-_l $ for all $l$; and
   \item $\widetilde{v}  \circ \tau = \widetilde{u}$.
\end{itemize}
The translation in the $s$-direction induces an $\mathbb{R}$-action on $\widetilde{\mathcal{M}}(\gamma^+,\bs\gamma^-;J)$ and we write
$${\mathcal{M}}(\gamma^+,\bs\gamma^-;J)=\widetilde{\mathcal{M}}(\gamma^+,\bs\gamma^-;J)/\R.$$

\s\n {\em Modifiers.}
We also introduce the modifier $\ast$ as in $\mathcal{M}^\ast(\star)$ to restrict the moduli space $\mathcal{M}(\star)$.  For example, when $\ast$ is $\op{ind}=k$, we restrict to curves of Fredholm index $k$, and when $\ast$ is $A$, we restrict to curves representing a relative homology class $A$.

\s
A similar definition holds when $(\widehat W,d\kappa)$ is an exact symplectic cobordism from a contact form $\alpha^+$ for $(M^+,\xi^+)$ to a contact form $\alpha^-$ for $(M^-,\xi^-)$ endowed with an almost complex structure $\widehat{J}\in \mathcal{J}(J^+,J^-)$, where $J^+ \in \mathcal{J}(\alpha^+)$ and $J^- \in \mathcal{J}(\alpha^-)$. If $\gamma^+$ is a Reeb orbit of $\alpha^+$ and $\bs\gamma^-=(\gamma^-_1,\dots,\gamma^-_n)$ is an ordered collection of Reeb orbits of $\alpha^-$, we define {\em finite energy holomorphic curves of \textbf{Type A} in $(\widehat{W},\widehat{J})$ asymptotic to $(\gamma^+,\bs\gamma^-)$,} along the same lines as above. We then define the moduli space ${\mathcal{M}}(\gamma^+,\bs\gamma^-;\widehat{J})$ of equivalence classes of finite energy holomorphic curves of \textbf{Type A} in  $(\widehat W,\widehat{J})$ asymptotic to $(\gamma^+,\bs\gamma^-)$.

\subsubsection{Curves of \textbf{Type B}} \label{subsubsection: B}

Let $\alpha$ be a contact form on $(M,\xi)$, $J\in \mathcal{J}(\alpha)$, and $\Lambda$ be a Legendrian submanifold of $(M,\xi)$. Let $\sigma^+$ be an $\alpha$-Reeb chord from $\Lambda$ to itself, $\bs\sigma^-=(\sigma^-_1,\dots,\sigma^-_n)$ be an ordered sequence of $\alpha$-Reeb chords from $\Lambda$ to itself, and $\bs\gamma^-=(\gamma^-_1,\dots,\gamma^-_m)$ be an ordered collection of Reeb orbits of $\alpha$.  Let $\mathbb{D}$ be the closed unit disk in $\C$.

\begin{defn}
A finite energy holomorphic curve of {\em \textbf{Type B} in $(\mathbb{R}\times M,J)$ with boundary in $\mathbb{R}\times \Lambda$ asymptotic to $(\sigma^+,\bs\sigma^-,\bs\gamma^-)$} is a triple $(\widetilde{u},{\bf x}^-,{\bf z}^-)$, where ${\bf x}^- = (x^-_1,\dots,x^-_n)$ is an ordered collection of distinct points $\not=1$ in $\partial \mathbb{D}$, ${\bf z}^-=(z^-_1,\dots,z^-_m)$ is an ordered collection of points in the interior of $\mathbb{D}$, and $\widetilde{u}: (\mathbb{D} \setminus (\{1\}\cup {\bf x}^- \cup {\bf z}^-), i_0) \to (\mathbb{R}\times M,J)$ is a finite energy $J$-holomorphic curve satisfying:
\begin{itemize}
  \item if we make one full turn in $S^1$ in the counterclockwise sense starting at $1$ we hit the points of ${\bf x}^-$ in the order $x^-_1,\dots,x^-_n$,
  \item $\widetilde{u}$ is positively asymptotic to the Reeb chord $\sigma^+$ near the puncture $1$,
  \item $\widetilde{u}$ is negatively asymptotic to the Reeb chord $\sigma^-_l$ near $x^-_l$,
  \item $\widetilde{u}$ is negatively asymptotic to the Reeb orbit $\gamma^-_l$ near $z^-_l$,
  \item $\widetilde{u}(\partial \mathbb{D}  \setminus (\{1\}\cup x^-)) \subset \mathbb{R}\times \Lambda$.
\end{itemize}
\end{defn}

Let $\widetilde{\mathcal{M}}(\sigma^+,\bs\sigma^-,\bs\gamma^-;J)$ be the moduli space of equivalence classes of finite energy holomorphic curves of \textbf{Type B} in $(\mathbb{R}\times M,J)$ asymptotic to $(\sigma^+,\bs\sigma^-,\bs\gamma^-)$, where the equivalence relation is analogous to the one for \textbf{Type A}.  Again there is an $\mathbb{R}$-action on $\widetilde{\mathcal{M}}(\sigma^+,\bs\sigma^-,\bs\gamma^-;J)$ and we set
$${\mathcal{M}}(\sigma^+,\bs\sigma^-,\bs\gamma^-;J)=\widetilde{\mathcal{M}}(\sigma^+,\bs\sigma^-,\bs\gamma^-;J)/\R.$$

Next suppose that $(\widehat W,d\kappa)$ is an exact symplectic cobordism from a contact form $\alpha^+$ for $(M^+,\xi^+)$ to a contact form $\alpha^-$ for $(M^-,\xi^-)$, and $L \subset (\widehat W,d\kappa)$ is an SFT-admissible exact Lagrangian cobordism from a Legendrian $\Lambda^+$ in $(M^+,\xi^+)$ to a Legendrian $\Lambda^-$ in $(M^-,\xi^-)$. Let $\widehat{J}\in \mathcal{J}(J^+,J^-)$, where $J^+ \in \mathcal{J}(\alpha^+)$ and $J^- \in \mathcal{J}(\alpha^-)$. If $\sigma^+$ is  a $\alpha^+$-Reeb chord from $\Lambda^+$ to itself, $\bs\sigma^-=(\sigma^-_1,\dots,\sigma^-_n)$ is an ordered collection of $\alpha^-$-Reeb chords from $\Lambda^-$ to itself, and $\bs\gamma^-=(\gamma^-_1,\dots,\gamma^-_n)$ is an ordered collection of Reeb orbits of $\alpha^-$, then we can similarly define finite energy holomorphic curves of \textbf{Type B} in $(\widehat W, \widehat{J})$ with boundary in $L$ asymptotic to $(\sigma^+,\bs\sigma^-,\bs\gamma^-)$ and the moduli space ${\mathcal{M}}(\sigma^+,\bs\sigma^-,\bs\gamma^-;\widehat{J})$ of equivalence classes of finite energy holomorphic curves \textbf{Type B} in  $(\widehat W, \widehat{J})$ with boundary in $L$ asymptotic to $(\sigma^+,\bs\sigma^-,\bs\gamma^-)$.

\subsubsection{Curves of \textbf{Type C}} \label{subsubsection: C}

Let $\alpha$ be a contact form on $(M,\xi)$, $J\in \mathcal{J}(\alpha)$ and $(\Lambda,\widehat{\Lambda})$ be a pair of disjoint Legendrian submanifolds in $(M,\xi)$. Let $\tau^+$ and $\tau^-$ be $\alpha$-Reeb chords from $\Lambda$ to $\widehat{\Lambda}$, $\bs\sigma^-=(\sigma^-_1,\dots,\sigma^-_n)$ be an ordered sequence of $\alpha$-Reeb chords from $\Lambda$ to itself, $\widehat{\bs\sigma}^-=(\widehat{\sigma}^-_1,\dots,\widehat{\sigma}^-_{\widehat{n}})$ be an ordered sequence of $\alpha$-Reeb chords from $\widehat{\Lambda}$ to itself, and $\bs\gamma^-=(\gamma^-_1,\dots,\gamma^-_m)$ be an ordered collection of Reeb orbits of $\alpha$.

\begin{defn}
A holomorphic curve of {\em \textbf{Type C} in $(\mathbb{R}\times M,J)$ with boundary in $(\mathbb{R}\times \Lambda,\mathbb{R}\times \widehat{\Lambda})$ asymptotic to $(\tau^+,\tau^-,\bs\sigma^-,\widehat{\bs\sigma}^-,\bs\gamma^-)$} is a quadruple $(\widetilde{u},{\bf x}^-,\widehat{\bf x}^-,{\bf z}^-)$, where ${\bf z}^-=(z^-_1,\dots,z^-_m)$ is an ordered collection of distinct points in the interior of $\mathbb{D}$, ${\bf x}^-=(x^-_1,\dots,x^-_n)$ is an ordered collection of distinct points in $\partial\mathbb{D}$ which does not contain $1$ or $-1$ ,  $\widehat{\bf x}^-=(\widehat{x}^-_1,\dots,\widehat{x}^-_{\widehat{n}})$ is an ordered collection of distinct points in $\partial\mathbb{D}$ which is disjoint from ${\bf x}^-$ and does not contain $1$ or $-1$, and $\widetilde{u}: \mathbb{D} \setminus (\{-1,1\}\cup {\bf x}^- \cup \widehat{\bf x}^- \cup {\bf z}^-) \to (\mathbb{R} \times M,J)$ is a holomorphic curve satisfying:
\begin{itemize}
  \item all the points of ${\bf x}^-$ are in the upper hemisphere $S^+$ of the circle $\partial \mathbb{D}$, and if we move along $S^+$ in the counterclockwise direction starting at $1$ and ending at $-1$ we hit the points of ${\bf x}^-$ precisely in the order given by ${\bf x}^-$,
  \item all the points of $\widehat{\bf x}^-$ are in the lower hemisphere $S^-$ of the circle $\partial \mathbb{D}$, and if we move along $S^-$ in the counterclockwise direction starting at $-1$ and ending at $1$ we hit the points of $\widehat{\bf x}^-$ precisely in the order given by $\widehat{\bf x}^-$,
  \item $\widetilde{u}$ is positively asymptotic to $\tau^+$ near the puncture $1$,
  \item $\widetilde{u}$ is negatively asymptotic to $\tau^-$ near the puncture $-1$,
  \item $\widetilde{u}$ is negatively asymptotic to $\sigma^-_l$ near the puncture $x^-_l$,
  \item $\widetilde{u}$ is negatively asymptotic to $\widehat{\sigma}^-_l$ near the puncture $\widehat{x}^-_l$,
  \item $\widetilde{u}$ is negatively asymptotic to $\gamma^-_l$ near the puncture $z^-_l$,
  \item $\widetilde{u}(S^+ \setminus \{-1,1\}\cup {\bf x}^-)$ is contained in $\mathbb{R}\times \Lambda$,
    \item $\widetilde{u}(S^- \setminus \{-1,1\}\cup \widehat{\bf x}^-)$ is contained in $\mathbb{R}\times \widehat{\Lambda}$.
\end{itemize}
\end{defn}

Using the same recipe as above we define the moduli space $\widetilde{\mathcal{M}}(\tau^+,\tau^-,\bs\sigma^-,\widehat{\bs\sigma}^-,\bs\gamma^-;J)$ of equivalence classes of finite energy holomorphic curves of \textbf{Type C} in $(\mathbb{R}\times M,J)$ asymptotic to $(\tau^+,\tau^-,\bs\sigma^-,\widehat{\bs\sigma}^-,\bs\gamma^-)$ and $${\mathcal{M}}(\tau^+,\tau^-,\bs\sigma^-,\widehat{\bs\sigma}^-,\bs\gamma^-;J)=\widetilde{\mathcal{M}}(\tau^+,\tau^-,\bs\sigma^-,\widehat{\bs\sigma}^-,\bs\gamma^-;J)/\R.$$

Next suppose that $(\widehat W, d\kappa)$ is an exact symplectic cobordism from a contact form $\alpha^+$ for $(M^+,\xi^+)$ to a contact form $\alpha^-$ for $(M^-,\xi^-)$ and  $L \subset (\widehat W,d\kappa)$ (resp.\ $\widehat L$) is an SFT-admissible exact Lagrangian cobordism from a Legendrian $\Lambda^+$ (resp.\ $\widehat{\Lambda}^+$) in $(M^+,\xi^+)$ to a Legendrian $\Lambda^-$ (resp.\ $\widehat{\Lambda}^-$) in $(M^-,\xi^-)$. Let $\widehat{J}\in \mathcal{J}(J^+,J^-)$, where $J^+ \in \mathcal{J}(\alpha^+)$ and $J^- \in \mathcal{J}(\alpha^-)$. If $\tau^+$ is  an $\alpha^+$-Reeb chord from $\Lambda^+$ to $\widehat{\Lambda}^+$, $\tau^-$ is an $\alpha^-$-Reeb chord from $\Lambda^-$ to $\widehat{\Lambda}^-$, $\bs\sigma^-=(\sigma^-_1,\dots,\sigma^-_n)$ is an ordered sequence of $\alpha^-$-Reeb chords from $\Lambda^-$ to itself, $\widehat{\bs\sigma}^-=(\widehat{\sigma}^-_1,\dots,\widehat{\sigma}^-_{\widehat{n}})$ is an ordered sequence of $\alpha^-$-Reeb chords from $\widehat{\Lambda}^-$ to itself, and $\bs\gamma^-=(\gamma^-_1,\dots,\gamma^-_m)$ is an ordered collection of Reeb orbits of $\alpha^-$, we can define finite energy holomorphic curves of \textbf{Type C} in $(\widehat W,d\kappa, \widehat{J})$ with boundary in $L\cup \widehat{L}$ asymptotic to $(\tau^+,\tau^-,\bs\sigma^-,\widehat{\bs\sigma}^-,\bs\gamma^-)$.
We then define the moduli space ${\mathcal{M}}(\tau^+,\tau^-,\bs\sigma^-,\widehat{\bs\sigma}^-,\bs\gamma^-;\widehat{J})$  of equivalence classes of finite energy holomorphic curves of \textbf{Type C} in $(\widehat W,\widehat{J})$ with boundary in $L\cup \widehat{L}$ asymptotic to $(\tau^+,\tau^-,\bs\sigma^-,\widehat{\bs\sigma}^-,\bs\gamma^-)$.

\subsubsection{Compactification of moduli spaces}

By the results of \cite{CPT}, a moduli space $\mathcal{M}$ of {\bf Type A, B} or {\bf C} admit a natural compactification $\overline{\mathcal{M}}$ called {\em the SFT compactification.} The compactification $\overline{\mathcal{M}}$ consists not only of pseudoholomorphic curves but also of multiple-level pseudoholomorphic buildings in the sense of \cite{CPT}. Pseudoholomorphic buildings are collections of pseudoholomorphic curves which satisfy certain matching conditions. We refer the reader to \cite{CPT} for the precise definitions.

\s
{\em From now on we restrict attention to the case where the contact manifolds are $3$-dimensional and the symplectic cobordisms are $4$-dimensional.}

\subsection{Full contact homology} \label{sec3.2.1}

The full contact homology of a contact manifold, introduced in \cite{SFT}, is an important invariant of contact structures. We refer the reader to \cite{SFT} and \cite{B} for detailed presentations of the material contained in this subsection.

Let $(M,\xi)$ be a contact $3$-manifold, $\alpha$ a nondegenerate contact form for $\xi$, and $R_\alpha$ the Reeb vector field for $\alpha$. We denote by $\mathcal{P}(\alpha)$ the set of good contractible periodic orbits of $R_\alpha$. To each orbit $\gamma \in \mathcal{P}(\alpha)$, we assign a $\mathbb{Z}/2$-grading $|\gamma| = \mu_{CZ}(\gamma) -1 \mbox{ mod } 2$. An orbit $\gamma$ is called {\em good} if it is either simple, or if $\gamma = (\gamma')^i$ for a simple orbit $\gamma'$ with $ |\gamma|=|\gamma'|$.

We will be assuming the following:

\begin{hypH}
There exists a perturbation scheme ${\frak P}$ which consists of $J\in \mathcal{J}(\alpha)$ and an assignment of a transversely cut out $\Q$-weighted branched manifold $\mathcal{Z}^A(\gamma,\bs\gamma')$ to each compactified moduli space $\overline{\mathcal{M}^A}(\gamma,\bs\gamma';J)$ which satisfies the following:
\be
\item if $\overline{\mathcal{M}^A}(\gamma,\bs\gamma';J)=\mathcal{M}^A(\gamma,\bs\gamma';J)$, is transversely cut out, and has Fredholm index $\op{ind}=1$, then $\# \overline{\mathcal{M}^A}(\gamma,\bs\gamma';J)=\#\mathcal{Z}^A(\gamma,\bs\gamma')$, where $\#$ refers to the weighted algebraic count;
\item if $\mathcal{Z}^A(\gamma,\bs\gamma')$ has $\op{ind}=2$, then
$$\bdry\mathcal{Z}^A(\gamma,\bs\gamma')=\coprod\mathcal{Z}^{A_1}(\gamma,\bs\gamma_1) \times \mathcal{Z}^{A_2}(\gamma_2,\bs\gamma_3),$$
where the disjoint union is over all pairs of $\op{ind}=1$ moduli spaces that formally glue to yield a curve from $\gamma$ to $\bs\gamma'$.
\ee
\end{hypH}

Hypothesis H has now been proven by Bao-Honda \cite{BH}; Pardon \cite{Pa1,Pa2} gives an algebraic substitute. Establishing Hypothesis H could also be done using the polyfold technology of Hofer-Wysocki-Zehnder~\cite{HWZ2} or the Kuranishi structures of Fukaya-Oh-Ohta-Ono~\cite{FO3}.

We now describe the contact homology differential graded algebra (dga) $\mathfrak{A}(M,\alpha, {\frak P})$. As an algebra $\mathfrak{A}(M,\alpha, {\frak P})$ is the graded commutative $\mathbb{Q}$-algebra with unit generated by $\mathcal{P}(\alpha)$. The commutativity means that the relation $a b = (-1)^{|a||b|}ba$ is valid for all $a,b \in \mathfrak{A}(M,\alpha,{\frak P})$. The $\mathbb{Z}/2$-grading on the elements of the algebra is obtained by considering on the generators the grading mentioned above and extending it to $\mathfrak{A}(M,\alpha, {\frak P})$.

We define the differential $\bdry$ on $\gamma\in \mathcal{P}(\alpha)$ as follows:
\begin{equation}
\bdry \gamma = m(\gamma) \sum_{\bs{\gamma}'=(\gamma'_1,\dots,\gamma'_m)} A(\gamma,\bs\gamma') \gamma'_1 \gamma'_2\dots \gamma'_m,
\end{equation}
where the sum is taken over all finite ordered collections $\bs\gamma'$ of elements of $\mathcal{P}(\alpha)$, $A(\gamma,\bs\gamma')\in \Q$ is the algebraic count of points in $\mathcal{Z}^{\op{ind}=1}(\gamma,\bs\gamma')$ and $m(\gamma)$ is the multiplicity of $\gamma$.  The map $\bdry$ is extended to the whole algebra by the graded Leibniz rule.

Part (2) of Hypothesis H implies that $\bdry^2=0$. Therefore, $\mathfrak{A}(M,\alpha,{\frak P})$ is a $\mathbb{Z}/2$-graded commutative dga.

\begin{defn}
The {\em full contact homology} $HC(M,\alpha,{\frak P})$ is the homology of the dga $\mathfrak{A}(M,\alpha,{\frak P})$.
\end{defn}

The full contact homology is independent of the choice of contact form $\alpha$ for $(M,\xi)$, the choice of the cylindrical almost complex structure $J\in \mathcal{J}(\alpha)$, and the choice of abstract perturbation scheme ${\frak P}$.

\subsection{The Legendrian contact homology of a Legendrian knot}

Let $\Lambda$ be a Legendrian knot in $(M,\xi)$ and let $\alpha$ be a nondegenerate, $\Lambda$-nondegenerate contact form for $(M,\xi)$.  Let $\mathcal{T}_\alpha(\Lambda)$ be the set of $\alpha$-Reeb chords starting and ending at $\Lambda$, and that are trivial in $\pi_1(M,\Lambda)$.

Using a perturbation scheme ${\frak P}$ which extends that of full contact homology and which will usually be suppressed from the notation, one similarly associates to $(\alpha,\Lambda)$ a differential graded algebra $\mathfrak{C}(M,\alpha,\Lambda)$. This construction is due to Eliashberg, Givental and Hofer and was outlined in \cite[Section 2.8]{SFT}. The algebra $\mathfrak{C}(M,\alpha,\Lambda)$ is the $\mathbb{Z}/2$-graded associative unital algebra generated by elements of $\mathcal{T}_\alpha(\Lambda)$ and  $\mathcal{P}(\alpha)$ with the relations
\begin{eqnarray}
\gamma_1 \gamma_2 = (-1)^{|\gamma_2||\gamma_1| } \gamma_2 \gamma_1, \\
\gamma_1 \sigma = (-1)^{|\sigma||\gamma_1|} \sigma \gamma_1,
\end{eqnarray}
for all $\gamma_1,\gamma_2 \in \mathcal{P}(\alpha)$ and $\sigma \in \mathcal{T}_\alpha(\Lambda)$. Equivalently, we can think of $\mathfrak{C}(M,\alpha,\Lambda)$ as the $\mathbb{Z}/2$-graded associative unital algebra generated by elements of $\mathcal{T}_\alpha(\Lambda)$ with coefficients in $\mathfrak{A}(M,\alpha)$.
It is then clear that $\mathfrak{C}(M,\alpha,\Lambda)$ is a right and left module over the differential graded algebra $\mathfrak{A}(M,\alpha)$. The degree $|\sigma|$ of an element $\sigma\in \mathcal{T}_\alpha(\Lambda)$ is given by the parity of its Conley-Zehnder index minus one, and is again extended algebraically to the whole algebra.

We define a differential $\partial_{\Lambda}$ on the generators $\sigma$ of $\mathcal{T}_\alpha(\Lambda)$ as follows:
\begin{equation}\label{diffonelegendrian}
\partial_{\Lambda} \sigma = \sum_{\bs\sigma'=(\sigma'_1,\dots,\sigma'_{k}),\bs\gamma=(\gamma_1,\dots,\gamma_m)} \sigma'_1\dots\sigma'_{k} \widetilde{C}(\sigma,\bs\sigma,\bs\gamma)\gamma_1\dots\gamma_m,
\end{equation}
where the sum is taken over all finite ordered collections $\bs\sigma'$ of elements of $\mathcal{T}_\alpha(\Lambda)$ and all finite ordered collections $\bs\gamma$ of elements of $\mathcal{P}(\alpha)$, $\widetilde{C}(\sigma,\bs\sigma',\bs\gamma) \in \mathbb{Q}$ is the algebraic count of the perturbation $\mathcal{Z}^{\op{ind}=1}(\sigma,\bs\sigma',\bs\gamma)$ of $\overline{\mathcal{M}}^{\op{ind}=1}(\sigma,\bs\sigma',\bs\gamma;J)$ and $m(\gamma_j)$ is the multiplicity of the Reeb orbit $\gamma_j$.  The differential is extended using the graded Leibniz rule. In particular, if $\gamma \in \mathfrak{A}(M,\alpha)$ and $\sigma \in \mathcal{T}_\alpha(\Lambda)$ then
\begin{equation}\label{eqcoefficient1}
\partial_{\Lambda} (\sigma\gamma) = \partial_{\Lambda}(\sigma)\gamma  +  (-1)^{|\sigma|}\sigma\partial(\gamma).
\end{equation}

The Legendrian contact homology analog of Hypothesis H and the outline given in \cite[Section 2.8]{SFT} imply that $\partial_{\Lambda}^2=0$.

\subsection{The opposite of $(\mathfrak{C}(M,\alpha,\Lambda),\partial_{\Lambda})$}

We recall the notion of the opposite of  $(\mathfrak{C}(M,\alpha,\Lambda),\partial_{\Lambda})$; cf. \cite{Halp,St} for the definition for general differential graded algebras.
First we define for all elements  $a,b \in \mathfrak{C}(M,\alpha,\Lambda)$ the product
\begin{equation}
a \bigcdot_{op} b = (-1)^{|a||b|}ba.
\end{equation}
The algebra obtained by considering the addition and multiplication by scalars as in $\mathfrak{C}(M,\alpha,\Lambda)$ and the product $\bigcdot_{op}$ is called the {\em opposite of $\mathfrak{C}(M,\alpha,\Lambda)$} and is denoted by $\mathfrak{C}_{op}(M,\alpha,\Lambda)$. We also consider the grading for $\mathfrak{C}_{op}(M,\alpha,\Lambda)$ to be the same as that for $\mathfrak{C}(M,\alpha,\Lambda)$.

On $\mathfrak{C}_{op}(M,\alpha,\Lambda)$ we consider the differential $\partial_\Lambda$ defined on $\mathfrak{C}(M,\alpha,\Lambda)$. A direct computation shows that $\partial_\Lambda$ is a differential on $\mathfrak{C}_{op}(M,\alpha,\Lambda)$ since it satisfies the graded Leibniz rule
\begin{equation}
\partial_\Lambda(a \bigcdot_{op} b)= (\partial_\Lambda a) \bigcdot_{op} b + (-1)^{|a|}a \bigcdot_{op} \partial_\Lambda b
\end{equation}
for all $a, b \in \mathfrak{C}_{op}(M,\alpha,\Lambda)$. The differential graded algebra $(\mathfrak{C}_{op}(M,\alpha,\Lambda),\partial_\Lambda)$ is called the {\em opposite of $(\mathfrak{C}(M,\alpha,\Lambda),\partial_{\Lambda})$.}

\subsection{The Legendrian contact homology of a pair $\Lambda \rightarrow\widehat{\Lambda}$}

The homology we present in this subsection is the one described in \cite[Remark 7.4]{BEE}; a similar construction appears in \cite{Ek}.
Let $\Lambda$ and $\widehat{\Lambda}$ two disjoint Legendrian knots on $(M,\xi)$. Suppose that $\alpha$ is nondegenerate and $\Lambda\cup\widehat{\Lambda}$-nondegenerate. Let $\mathcal{T}_\alpha(\Lambda \rightarrow\widehat{\Lambda})$ be the set of $\alpha$-Reeb chords with initial point on $\Lambda$ and terminal point on $\widehat{\Lambda}$. Again we use a perturbation scheme ${\frak P}$ which will usually be suppressed from the notation.

We first consider the dga $\mathfrak{C}_{op}(M,\alpha,\Lambda) \otimes_{\mathfrak{A}(M,\alpha)} \mathfrak{C}(M,\alpha,\widehat{\Lambda})$ whose differential is the Koszul differential of the tensor product. Let $\mathfrak{B}(M,\alpha,\Lambda \rightarrow\widehat{\Lambda})$ be the $\mathbb{Z}/2$-graded right module over $\mathfrak{C}_{op}(M,\alpha,\Lambda) \otimes_{\mathfrak{A}(M,\alpha)} \mathfrak{C}(M,\alpha,\widehat{\Lambda})$ generated by $\mathcal{T}_\alpha(\Lambda \rightarrow\widehat{\Lambda})$.
The differential $\partial_{\Lambda \rightarrow\widehat{\Lambda}}$ in $\mathfrak{B}(M,\alpha,\Lambda \rightarrow\widehat{\Lambda})$ is defined on elements
$\tau \in \mathcal{T}_\alpha(\Lambda \rightarrow\widehat{\Lambda})$ by:
\begin{equation}\label{diffpair}
\partial_{\Lambda \rightarrow\widehat{\Lambda}} \tau = \sum_{\tau' \in \mathcal{T}_\alpha(\Lambda \rightarrow\widehat{\Lambda})}  \tau' B(\tau,\tau').
\end{equation}
Here $B(\tau,\tau')\in \mathfrak{C}_{op}(M,\alpha,\Lambda) \otimes_{\mathfrak{A}(M,\alpha)} \mathfrak{C}(M,\alpha,\widehat{\Lambda})$ is given by
\begin{equation}\label{coeffpair}
B(\tau,\tau') = \sum
_{\bs\sigma,\widehat{\bs\sigma}^-,\bs\gamma} \sigma_1 \dots\sigma_n\widehat{\sigma}^-_1,\dots,\widehat{\sigma}^-_{\widehat{n}}
\gamma_1\dots\gamma_m \widetilde{B}(\tau,\tau',\bs\sigma,\widehat{\bs\sigma}^-,\bs\gamma),
\end{equation}
where the sum is taken over all ordered finite collections $\bs\sigma=(\sigma_1,\dots,\sigma_n)$ of elements in $\mathcal{T}_\alpha(\Lambda)$, all ordered finite collections $\widehat{\bs\sigma}^-=(\widehat{\sigma}^-_1,\dots,\widehat{\sigma}^-_{\widehat{n}})$ of elements in $\mathcal{T}_\alpha(\widehat{\Lambda})$, and all ordered finite collections $\bs\gamma=(\gamma_1,\dots,\gamma_m)$ of elements in $\mathcal{P}(\alpha)$, and $\widetilde{B}(\tau,\tau',\bs\sigma,\widehat{\bs\sigma}^-,\bs\gamma)$ is the algebraic count of the perturbation $\mathcal{Z}^{\op{ind}=1}(\tau,\tau',\bs\sigma, \widehat{\bs\sigma}^-,\bs\gamma)$ of $\overline{\mathcal{M}}^{\op{ind}=1}(\tau,\tau',\bs\sigma, \widehat{\bs\sigma}^-, \bs\gamma)$.

An analog of Hypothesis H and the outline given in \cite[Section 2.8]{SFT} imply that $\partial_{\Lambda \rightarrow\widehat{\Lambda}}^2 = 0$.

\begin{rmk}
The reason we have to use the opposite dga~$(\mathfrak{C}_{op}(M,\alpha,\Lambda),\partial_\Lambda)$ is that for geometric reasons elements of $\mathcal{T}_\alpha(\Lambda)$ have to be multiplied to the left of elements of $\mathcal{T}_\alpha(\Lambda \rightarrow\widehat{\Lambda})$ in the module $\mathfrak{B}(M,\alpha,\Lambda \rightarrow\widehat{\Lambda})$. This implies that $\mathfrak{B}(M,\alpha,\Lambda \rightarrow\widehat{\Lambda})$ must be a left module over $\mathfrak{C}(M,\alpha,{\Lambda})$ and thus a right module over $\mathfrak{C}_{op}(M,\alpha,{\Lambda})$.
Analogously elements of $\mathcal{T}_\alpha(\widehat{\Lambda})$ have to be multiplied to the right of elements of $\mathcal{T}_\alpha(\Lambda \rightarrow\widehat{\Lambda})$, which implies that $\mathfrak{B}(M,\alpha,\Lambda \rightarrow\widehat{\Lambda})$ must be a right module over $\mathfrak{C}(M,\alpha,\widehat{\Lambda})$.

Combining these observations we conclude that the correct choice is to construct $\mathfrak{B}(M,\alpha,\Lambda \rightarrow\widehat{\Lambda})$ as a right module over $\mathfrak{C}_{op}(M,\alpha,\Lambda) \otimes_{\mathfrak{A}(M,\alpha)} \mathfrak{C}(M,\alpha,\widehat{\Lambda})$.
\end{rmk}

\subsection{Augmentations and strip Legendrian contact homology} \label{subsection: linearizations}

The goal of this subsection is to explain augmentations and the linearization process, introduced to Legendrian contact homology by Chekanov \cite{Ch}.  This allows us to define a differential $\bdry^\epsilon_{\Lambda \rightarrow\widehat{\Lambda}}$ on the $\mathbb{Q}$-vector space $LC_{\epsilon}(\alpha, \Lambda \rightarrow\widehat{\Lambda})$ generated by $\mathcal{T}_\alpha(\Lambda \rightarrow\widehat{\Lambda})$ in the presence of an augmentation $\epsilon$.

An {\em augmentation $\epsilon_{\mathfrak{A}(M,\alpha)}$ for $\mathfrak{A}(M,\alpha)$} is a dga morphism $\mathfrak{A}(M,\alpha)\to\mathbb{Q}$ with the trivial differential for $\Q$, i.e.:
\begin{gather}
  \epsilon_{\mathfrak{A}(M,\alpha)}(q\mathbf{1})= q \mbox{ for every } q \in \mathbb{Q},\\
  \epsilon_{\mathfrak{A}(M,\alpha)} \circ \partial= 0,
\end{gather}
where $\mathbf{1}$ is the unit in $\mathfrak{A}(M,\alpha)$.

An {\em augmentation $\epsilon_{\mathfrak{C}(M,\alpha,\Lambda)}$ for $\mathfrak{C}(M,\alpha,\Lambda)$} is a dga morphism $\mathfrak{C}(M,\alpha,\Lambda)\to\mathbb{Q}$, i.e.:
\begin{gather}
  \epsilon_{\mathfrak{C}(M,\alpha,\Lambda)} (q \mathbf{1}) = q \mbox{ for all } q \in \mathbb{Q},\\
  \epsilon_{\mathfrak{C}(M,\alpha,\Lambda)} \circ \partial_{\Lambda} =0,
\end{gather}
where $\mathbf{1}$ is the unit in $\mathfrak{C}(M,\alpha,\Lambda)$. It is straightforward to check that an augmentation $\epsilon_{\mathfrak{C}(M,\alpha,\Lambda)}$ for $\mathfrak{C}(M,\alpha,\Lambda)$ is also an augmentation for the opposite dga $(\mathfrak{C}_{op}(M,\alpha,\Lambda),\partial_\Lambda)$.

\begin{defn}
Given augmentations $\epsilon_{\mathfrak{A}(M,\alpha)}$, $\epsilon_{\mathfrak{C}(M,\alpha,\Lambda)}$ and $\epsilon_{\mathfrak{C}(M,\alpha,\widehat{\Lambda})}$ that satisfy the following compatibility condition:
\begin{gather}
   \epsilon_{\mathfrak{C}(M,\alpha,\Lambda)}(a) = \epsilon_{\mathfrak{C}(M,\alpha,\widehat{\Lambda})} (a) = \epsilon_{\mathfrak{A}(M,\alpha)}(a),
\end{gather}
for every $a \in \mathfrak{A}(M,\alpha)$, the {\em composite augmentation}
\begin{equation} \label{eqn: total augmentation}
\epsilon:\mathfrak{C}_{op}(M,\alpha,\Lambda) \otimes_{\mathfrak{A}(M,\alpha)} \mathfrak{C}(M,\alpha,\widehat{\Lambda})\to\mathbb{Q}
\end{equation}
is a dga morphism that satisfies:
\begin{itemize}
\item $\epsilon$ coincides with $\epsilon_{\mathfrak{A}(M,\alpha)}$ on elements of the coefficient dga $\mathfrak{A}(M,\alpha)$,
\item $\epsilon$ coincides with $\epsilon_{\mathfrak{C}(M,\alpha,\Lambda)}$ on elements of the form $\mathfrak{C}_{op}(M,\alpha,\Lambda) \otimes \mathbf{1}$,
\item $\epsilon$ coincides with $\epsilon_{\mathfrak{C}(M,\alpha,\widehat{\Lambda})}$ on elements of the form $\mathbf{1} \otimes \mathfrak{C}(M,\alpha,\widehat{\Lambda})$,
\item $\epsilon(q\mathbf{1} \otimes \mathbf{1}) = q$ for all $q\in \mathbb{Q}$.
\end{itemize}
\end{defn}

We are now ready to define our differential $\bdry^\epsilon_{\Lambda \rightarrow\widehat{\Lambda}}$ on the $\mathbb{Q}$-vector space $LC_{\epsilon}(\alpha, \Lambda \rightarrow\widehat{\Lambda})$. We define for each $\sigma \in \mathcal{T}_\alpha(\Lambda \rightarrow\widehat{\Lambda})$
\begin{equation}\label{maindiff}
\bdry^\epsilon_{\Lambda \rightarrow\widehat{\Lambda}} \sigma := \sum_{\sigma' \in \mathcal{T}_\alpha(\Lambda \rightarrow\widehat{\Lambda})} \epsilon(B(\sigma,\sigma')) \sigma'.
\end{equation}

It is a theorem of Bourgeois, Ekholm and Eliashberg that $(\bdry^\epsilon_{\Lambda \rightarrow\widehat{\Lambda}})^2 =0$.
To see this we first define a map $F_\epsilon : \mathfrak{B}(M,\alpha,\Lambda \rightarrow\widehat{\Lambda}) \to LC_{\epsilon}(\alpha, \Lambda \rightarrow \widehat{\Lambda})$ by setting $\sigma \in \mathcal{T}_\alpha(\Lambda \rightarrow \widehat{\Lambda})$ and $\widetilde{C} \in \mathfrak{C}(M,\alpha,\Lambda) \otimes_{\mathfrak{A}(M,\alpha)} \mathfrak{C}(M,\alpha,\widehat{\Lambda})$:
\begin{equation}\label{mapFepsilon}
F_\epsilon (\sigma \widetilde{C}) = \epsilon(\widetilde{C}) \sigma.
\end{equation}
A direct computation shows that the diagram \\
\[
\begin{CD}
 \mathfrak{B}(M,\alpha,\Lambda \rightarrow \widehat{\Lambda})   @>{F_\epsilon}>>  LC_{\epsilon}(\alpha, \Lambda \rightarrow \widehat{\Lambda}) \\
 @V{\partial_{\Lambda \rightarrow \widehat{\Lambda}}}VV          @V{\bdry^\epsilon_{\Lambda \rightarrow \widehat{\Lambda}}}VV \\
 \mathfrak{B}(M,\alpha,\Lambda \rightarrow \widehat{\Lambda})  @>{F_\epsilon}>>  LC_{\epsilon}(\alpha, \Lambda \rightarrow \widehat{\Lambda})
\end{CD}
\]
is commutative. It then follows directly from the fact that $(\partial_{\Lambda \rightarrow \widehat{\Lambda}})^2=0$ that $(\bdry^\epsilon_{\Lambda \rightarrow \widehat{\Lambda}})^2 =0$.

In the special case that
\begin{itemize}
\item there is no index one holomorphic plane asymptotic to a periodic orbit of $R_\alpha$ and
\item there is no index one once-punctured holomorphic disk asymptotic to a Reeb chord for $\Lambda$ or $\widehat{\Lambda}$,
\end{itemize}
the trivial map $\epsilon$ which vanishes for all Reeb chords and Reeb orbits and restricts to the identity on $\mathbb{Q}$ is an augmentation. The corresponding linearized Legendrian contact homology $LCH_\epsilon (\alpha, \Lambda \rightarrow \widehat{\Lambda})$ whose differential counts pseudoholomorphic strips is called the {\em strip Legendrian contact homology of $\Lambda \rightarrow \widehat{\Lambda}$}. It is denoted
$LCH_{st} (\alpha, \Lambda \rightarrow \widehat{\Lambda})$.

\subsection{Pullback augmentations and linearizations}

As shown in \cite{BEE} one can use exact symplectic cobordisms to pull back linearizations. We explain this in more detail.

Let $(\mathbb{R}\times M,d\sigma)$ be an exact symplectic cobordism from a contact form $\alpha$ for $(M,\xi)$ to a contact form $\alpha^0$ for $(M,\xi)$. Let $L,\widehat{L} \subset (\mathbb{R}\times M,d\sigma)$ be disjoint, cylindrical SFT-admissible exact Lagrangian cobordisms from Legendrian knots $\Lambda,\widehat{\Lambda}$ to themselves.

The cobordisms $(\R\times M,d\sigma,L)$ and $(\R\times M,d\sigma,\widehat{L})$ induce chain maps
$$\Psi_{\R\times M,d\sigma,L}:\mathfrak{C}(M,\alpha,\Lambda) \to \mathfrak{C}(M,\alpha^0,\Lambda),$$
$$\Psi_{\R\times M, d\sigma,\widehat{L}}: \mathfrak{C}(M,\alpha,\widehat\Lambda) \to \mathfrak{C}(M,\alpha^0,\widehat\Lambda)$$
By the arguments of \cite{BEE,SFT}, the chain maps $\Psi_{\R\times M,d\sigma,L}$ and $\Psi_{\R\times M,d\sigma,\widehat{L}}$ are quasi-isomorphisms. It is straightforward to check that the fact that $\Psi_{\R\times M,d\sigma,L}:\mathfrak{C}(M,\alpha,\Lambda) \to \mathfrak{C}(M,\alpha^0,\Lambda)$ is a quasi-isomorphism implies that $\Psi_{\R\times M,d\sigma,L}:\mathfrak{C}_{op}(M,\alpha,\Lambda) \to \mathfrak{C}_{op}(M,\alpha^0,\Lambda)$ is  a quasi-isomorphism.

If $$\epsilon^0 : \mathfrak{C}_{op}(M,\alpha^0,\Lambda)\otimes_{\mathfrak{A}(M,\alpha^0)} \mathfrak{C}(M,\alpha^0,\widehat{\Lambda}) \to \mathbb{Q}$$
is a composite augmentation, then we can define a composite augmentation
$$\epsilon^0 \circ (\Psi_{\R\times M,d\sigma,L}\otimes \Psi_{\R\times M, d\sigma,\widehat{L}}): \mathfrak{C}_{op}(M,\alpha,\Lambda)\otimes_{\mathfrak{A}(M,\alpha)} \mathfrak{C}(M,\alpha,\widehat{\Lambda}) ) \to \mathbb{Q}.$$

\subsection{Action filtration}

There is an action filtration on contact homology and Legendrian contact homology.

To see this we first note that if $J\in \mathcal{J}(\alpha)$ and
$$\widetilde{u}: (S\setminus \Omega,i) \to (\mathbb{R}\times M,J)$$
is a finite energy $J$-holomorphic curve in the symplectization $(\mathbb{R}\times  M,d (e^s\alpha))$, then $\int_{S} \widetilde{u}^* d\alpha \geq 0$ and moreover $\int_{S} \widetilde{u}^* d\alpha=0$ if and only if $\widetilde{u}$ is either a trivial strip over a Reeb chord or a trivial cylinder over a Reeb orbit; see \cite{Ab,H}.

Let $(z^+_1,\dots,z^+_n)$ be the positive punctures of $\widetilde{u}$ and $(z^-_1,\dots,z^-_m)$ be the negative punctures of $\widetilde{u}$. If $A(z^{\pm}_j)$ is the action of the Reeb orbit or Reeb chord that $\widetilde{u}$ limits to near $z^{\pm}_j$, one obtains the following formula:
\begin{equation}
\int_{S} \widetilde{u}^* d\alpha = \sum_{j=1}^{n}A(z^{+}_j) - \sum_{j=1}^{m}A(z^{-}_j).
\end{equation}
It is immediate that $\mathcal{M}(\gamma^+,\bs\gamma^-;J)=\varnothing$ if $A(\gamma^+) \leq \sum_{j=1}^{n} A(\gamma^-_j) $ and analogous statements hold for ${\mathcal{M}}(\sigma^+,\bs\sigma^-,\bs\gamma^-;J)$ and ${\mathcal{M}}(\tau^+,\tau^-,\bs\sigma^-,\widehat{\bs\sigma}^-,\bs\gamma^-;J)$.

These considerations imply that:
\begin{itemize}
\item the subalgebra $\mathfrak{A}^{\leq T}(M,\alpha)$ of $\mathfrak{A}(M,\alpha)$ generated by the Reeb orbits of $\alpha$ with action $\leq T$ is a sub-dga of $(\mathfrak{A}(M,\alpha),\bdry)$,
\item the algebra $\mathfrak{C}^{\leq T}(M,\alpha,\Lambda)$ of $\mathfrak{C}(M,\alpha,\Lambda)$ generated by Reeb orbits and Reeb chords with action $\leq T$ is a sub-dg-algebra of $(\mathfrak{C}(M,\alpha,\Lambda),\partial_{\Lambda})$, and a module over $\mathfrak{A}^{\leq T}(M,\alpha)$,
\item the subset $\mathfrak{B}^{\leq T}(M,\alpha,\Lambda \rightarrow \widehat{\Lambda})$ of $\mathfrak{B}(M,\alpha,\Lambda \rightarrow \widehat{\Lambda})$ generated by Reeb orbits and Reeb chords with action $\leq T$ endowed with the differential $\partial_{\Lambda \to \widehat{\Lambda}}$ is a dg-module over $\mathfrak{C}_{op}^{\leq T}(M,\alpha,\Lambda) \otimes_{\mathfrak{A}^{\leq T}(M,\alpha)} \mathfrak{C}^{\leq T}(M,\alpha,\widehat{\Lambda})$.
\end{itemize}

In the same way as in Section~\ref{subsection: linearizations} we can consider augmentations on $(\mathfrak{A}^{\leq T}(M,\alpha),\bdry)$, $(\mathfrak{C}^{\leq T}(M,\alpha,\Lambda),\partial_{\Lambda})$ and $(\mathfrak{B}^{\leq T}(M,\alpha,\Lambda \rightarrow \widehat{\Lambda}),\partial_{\Lambda \rightarrow \widehat{\Lambda}})$.

\section{Constructions}\label{section: constructions}

The goal of this section is to construct the Legendrians $\Lambda$, $\widehat{\Lambda}$, and the contact forms $\alpha_{\epsilon, \epsilon'}$ that are used in the proof of Theorem~\ref{thm: main}.

\subsection{Review of \cite{CH2}}

We first recall some notation and the construction of a family of controlled contact forms from \cite{CH2}.
Let $(S,h)$ be a supporting open book decomposition of $(M,\xi )$ where:
\begin{itemize}
\item $\partial S$ is connected;
\item $h$ is homotopic to a pseudo-Anosov homeomorphism $\psi$ of $S$.
\end{itemize}

The singularities of the stable foliation $\F$ along $\partial S$ are $x_1,\dots ,x_n$ and the ones in the interior of $S$ are $y_1,\dots ,y_q$.   We first consider a small neighborhood $N(S)$ of $\partial S$ in $S$ such that the component of $\partial N(S)$ contained in the interior of $S$ is a concatenation of arcs that are alternately tangent to and transverse to $\F$. Let $a_i$ be the transverse arc of $\bdry N(S)$ corresponding to $y_i$. Then we look at the separatrices emanating from the interior singularities $y_j$ and cut them at the first moment they enter $N(S)$. The corresponding arcs from $y_j$ to $\partial N(S)$ are called $Q_{j1}', \dots, Q_{jm_j}'$. Recalling that $P_i$ is the prong emanating from $x_i$, we similarly let $P_i'$ be the first component of $P_i\cap (S-int(N(\bdry S)))$ that can be reached from the singular point $x_i$, traveling inside an arc (called $p_i$) of $P_i$.  The arc  $P_i'$ has endpoints on $int(a_i)$ and some $int(a_{i'})$. Now, for each $j$, we define $N(y_j)$ to be a sufficiently small neighborhood of $Q_{j1}'\cup \dots\cup Q_{jm_j}'$ in $S-int(N(\bdry S))$ such that $\bdry N(y_j)$ is a concatenation of arcs that are alternately tangent to and transverse to $\F$.

We define the subset
\begin{equation}
S'' =S -\cup_{1\leq j\leq q} ~int(N( y_j ))-  \cup_{1\leq i\leq n} ~int(N(P_i')) - int (N(\partial S )).
\end{equation}
Here the boundary of the neighborhoods $N(P_i')$ are also alternatively tangent to and transverse to $\F$.

By \cite[Claim 6.4]{CH2} we know that $\F \vert_{S''}$ is orientable.

In \cite[Proposition 6.3]{CH2}, the second and third authors construct a $1$-form $\beta$ on $S$ which has the following properties:
\begin{itemize}
\item[(P1)] $\ker \beta=\F$ on $S''$;
\item[(P2)] $\beta$ is nonsingular (i.e., has no zeros) away from $N(\partial S)$ and from a very small neighborhood of the interior singularities $y_i$; 
\item[(P3)] $\beta$ has an elliptic singularity in each component of $N(\partial S)-\cup_{1\leq j\leq n} p_i$.
\end{itemize}
An important point in what follows is that
\begin{itemize}
\item[(P4)] a nonempty subset of the leaves from each elliptic singularity $e$ in $N(\partial S)-\cup_{1\leq j\leq n} p_i$ go to the boundary $\partial S$.
\end{itemize}
This means that we can take the $d\beta$-area of a small disk around each elliptic singularity $e$ to be arbitrary large without affecting the value of $\beta$ in $S-N(\partial S)$.

\subsection{Construction of $\Lambda$ and $\widehat\Lambda$} \label{subsection: construction of Lambda}

Recall the Rademacher function $\Phi$ from Section~\ref{subsection: Rademacher function}. The following is the main technical lemma:

\begin{lemma}\label{lemma: technical}
There exists a $1$-form $\beta$ as above such that there exists an immersed curve $L$ in $S$ such that:
\begin{itemize}
\item[(A)] $\int_{L} \beta =0$;
\item[(B)] for every immersed arc $\delta$ in $L$, we have $\Phi (\delta ) =0$;
\item[(C)] every closed curve $\delta$ which is $\pi_1$-injective in $L$ and which turns at most once at a double point of $L$ is $\pi_1$-injective in $S$.
\end{itemize}
\end{lemma}

\begin{proof}
We claim that there exists an oriented embedded closed curve $L_0 =\delta_0 \cup \delta_1$ (see Figure~\ref{figure: L_0}) in $S$ such that:
\begin{enumerate}
\item[(L1)] $\delta_0$ is a sufficiently short arc transverse to $\F$, $\delta_1$ is a (long) arc tangent to $\F$, and $int(\delta_0)\cap int(\delta_1)=\varnothing$ \item[(L2)] $\delta_1$ is not contained in a separatrix;
\item[(L3)] $\delta_1$ is positively (or negatively) transverse to $\delta_0$ at its two boundary points and the sign is the same for both;
\item[(L4)] every closed curve on $S$ which is the union of a nontrivial embedded subarc of $\delta_0$ and an arc tangent to $\mathcal{F}$ is $\pi_1$-injective in $S$.
\end{enumerate}
\begin{figure}[ht]
\begin{center}
\begin{overpic}[width=4cm]{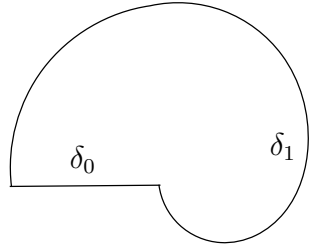}
\put(20,25.5){$\delta_0$}  \put(86.5,30){$\delta_1$}
\end{overpic}
\end{center}
\caption{The curve $L_0$} \label{figure: L_0}
\end{figure}
This follows from the density of the leaves of the stable foliation $\F$ and we briefly sketch the idea. Let $r_0\in int(S)$ be a point not on a separatrix and let $\mathcal{L}$ be a half-leaf of $\F$ starting from $r_0$.  We take a short arc $c$ that is transverse to $\F$ and has $r_0$ as an endpoint and consider the first point of return $r_1$ of $\mathcal{L}$ to $c$.  If $\mathcal{L}$ and $c$ intersect with the same sign at both $r_0$ and $r_1$, then we are done.  Otherwise, consider the second point of return $r_2$ of $\mathcal{L}$ to the subarc $[r_0,r_1]$ of $c$. Then either $r_0$ and $r_2$ have the same sign or $r_1$ and $r_2$ have the same sign, and we are done.

(L4) is due to the fact that leaves of $\F$ are quasi-geodesic for some hyperbolic metric on $S$ and $\delta_0$ is sufficiently short.

Next we apply the construction of $\beta$ in such a way that $N(\partial S)\cap L_0 =\emptyset$ and $\delta_1 \subset S''$. This can be done by taking $N(S)$, $\cup_{1\leq j\leq q} ~int(N( y_j ))$ and $ \cup_{1\leq i\leq n} ~int(N(P_i'))$ to be sufficiently small.  (B) and (C) hold for $L_0$ but
$\int_{L_0} \beta =\int_{\delta_{0}} \beta$ has no reason to vanish.

We will now modify $L_0$ to $L$ so that $\int_L\beta=0$ while preserving conditions (B) and (C); see Figure~\ref{figure: modification}. We extend $\delta_1$ at both ends by arcs $a$ and $a'$ that are tangent to $\F$ until they hit $N(\partial S)$ and then go inside $N(\partial S)$ to the elliptic singularity $e$ or $e'$ of $\ker \beta$
contained in the appropriate component of $N(\partial S)-\cup_{1\leq j\leq n} p_i$.
Notice here that $a$ and $a'$ might intersect $\delta_0$ at some interior points.
Depending on the sign of $\int_{L_0} \beta$, we take $L$ to be the result of applying a finger move to $L_0$ along $a$ or $a'$ so that it circles around $e$ or $e'$; say we are using $a$. (The orientation condition (L3) tells that the two possibilities contribute different signs.) Applying the finger move means we delete a small portion of $\delta_0$ from $L_0$ and replace by a long arc with the same endpoints that is tangent to $\F$ in $S''$ and that goes around $e$ in $N(\partial S )$ without crossing the $p_i$'s.
\begin{figure}[ht]
\begin{center}
\psfragscanon
\psfrag{a}{\small $\delta_1$}
\psfrag{c}{\small $\delta_0$}
\psfrag{d}{\small $a$}
\psfrag{e}{\small $e$}
\psfrag{f}{\small $\bdry S$}
\psfrag{g}{\small $L$}
\includegraphics[width=8cm,grid]{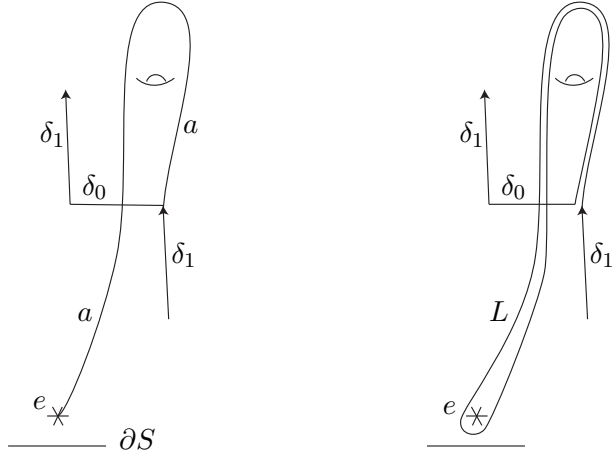}
\end{center}
\caption{Modification of $L_0$ along $a$ which gives $L$} \label{figure: modification}
\end{figure}
Since the $d\beta$-area around $e$ can be made arbitrarily large, we can use it to counterbalance $\int_{L_0} \beta$.  Hence it is possible to obtain $\int_L\beta=0$. We leave it to the reader to verify that (B) and (C) hold.  Note that the finger move creates bigons bounded by $L$ but not monogons.
\end{proof}

We lift the curve $L$ from Lemma~\ref{lemma: technical} to an embedded Legendrian curve $\Lambda$ in $(S\times [-c,c] ,dt+\epsilon \beta )$, where $t$ is the $[-c,c]$-coordinate. This is done by observing that the contact form $dt+\varepsilon \beta$ defines a connection for the trivial fibration $\pi : S\times \R \to S$.
The condition $\int_{L} \beta =0$ ensures that every lift of $L$ is a closed curve. Starting a lift at a point of altitude $t=0$, if $\varepsilon>0$ is sufficiently small, then the entire Legendrian lift $\Lambda$ of $L$ lies in the region $S\times [-c,c] \subset S\times \R$.

The Legendrian $\widehat{\Lambda}$ is taken to be a small pushoff of $\Lambda$ in the Reeb direction which is also contained in $S\times [-c,c] \subset S\times \R$.

\subsection{Construction of the contact forms $\alpha_{\epsilon,\epsilon'}$}\label{subsection: construction of forms}

The construction of the contact form $\alpha_{\epsilon, \epsilon'}$ and the corresponding Reeb vector field $R_{\epsilon,\epsilon'}$ then follows the lines of \cite[Section 6.2]{CH2}, where we are using the $1$-form $\beta$ satisfying Lemma~\ref{lemma: technical}. The only difference is that we plug into the mapping torus $(S\times\R)/(x,m)\sim(h(x),m-1)$ a copy of the $t$-invariant contact structure $(S\times [-c,c] , dt+\varepsilon \beta )$ for sufficiently small $c$ and $\varepsilon$. This part of the open book contains the Legendrians $\Lambda$ and $\widehat\Lambda$.

\section{Control of once-and twice-punctured holomorphic disks}

\begin{thm} \label{thm: once-punctured disks}
Suppose $k\geq 3$. Given $N\gg 0$, for sufficiently small $\epsilon,\epsilon'>0$, there are no once-punctured holomorphic disks in $\R\times M$ asymptotic to an $R_{\epsilon,\epsilon}$-Reeb chord of $\Lambda$ or $\widehat{\Lambda}$ of $\alpha_{\epsilon,\epsilon'}$-action $\leq N$ and with boundary on the symplectization of $\Lambda$ or $\widehat{\Lambda}$.
\end{thm}

\begin{proof}
The proof is the same as those of \cite[Theorems 8.1 and 9.2]{CH2}. We sketch it for $\Lambda$ (the proof is identical for $\widehat{\Lambda}$). Arguing by contradiction, suppose that such a once-punctured holomorphic disk $\widetilde u=(r,u)$ exists for $\Lambda$.

First suppose that the $M$-component $u$ does not intersect the binding $K$.  In this case the Reeb chord must correspond to a double point of $L$ and there is a projection of $u$ to $S$ which takes the boundary of $u$ to $L$.  However by (L4) no curve in $L$ that turns once at a double point is contractible, a contradiction.

On the other hand, if $u$ intersects $K$, then we remove a neighborhood $N(K)$ of $K$, choose a page which is disjoint from $\Lambda$ and cut $M-N(K)$ along this page. As explained in \cite[Sections 8 and 9]{CH2}, if we perform these cutting operations on $u$ and project to $S$ using the projection $\pi_S:S\times[0,1]\to S$ onto the first factor, then we obtain a disk $\mathcal{D}$.

We claim that $\Phi(\bdry\mathcal{D})\not=0$. Indeed, as in the absolute case $\bdry \mathcal{D}$ can be written as a concatenation of arcs, where the only difference here is that one of these arcs is the $\pi_S$-projection of an arc $d$ that lies in $\Lambda$. By construction, $\pi_S(d)\subset L$ and hence satisfies $\Phi(\pi_S(d))=0$. Since there is also one extra concatenation to consider, we lose an extra degree of precision and conclude that if $k\geq 3$ (instead of $2$), such an disk cannot exist.
\end{proof}

Since we already know from \cite[Theorem 8.1]{CH2} that, given $N\gg 0$, for sufficiently small $\epsilon,\epsilon'>0$ there are no holomorphic planes asymptotic to closed orbits of $R_{\epsilon,\epsilon'}$ of $\alpha_{\epsilon,\epsilon'}$-action $\geq N$ when $k\geq 2$, we conclude that the strip Legendrian contact homology of $\Lambda \rightarrow \widehat{\Lambda}$ is well-defined when $k\geq 3$.

\begin{thm} \label{thm: strips}
Suppose $k\geq 5$. Given $N\gg 0$, for sufficiently small $\epsilon,\epsilon'>0$, there are no holomorphic strips between $R_{\epsilon,\epsilon'}$-Reeb chords from $\Lambda$ to $\widehat{\Lambda}$ of $\alpha_{\epsilon,\epsilon'}$-action $\leq N$ which intersect $\R$ times the binding nontrivially.
\end{thm}

\begin{proof}
The proof is again same as that of \cite[Theorem 9.2]{CH2} that prevents holomorphic cylinders from intersecting the binding, except that we have to concatenate two extra arcs coming from the boundary components of the holomorphic strip that map to $\Lambda$ and $\widehat{\Lambda}$. Since $\Phi=0$ on these extra arcs by Lemma~\ref{lemma: technical}(B), the theorem follows without more effort at the cost of replacing the inequality $k\geq 3$ by $k\geq 5$ to take care of the possible extra leaking of $2$ due to the two extra concatenations.
\end{proof}

\section{Growth of the number of Reeb chords and positivity of topological entropy}

\subsection{A heuristic argument} \label{subsection: growth}

Let $(M,\xi)$ be a contact $3$-manifold which admits a supporting open book decomposition whose binding is connected and whose monodromy is homotopic to a pseudo-Anosov homeomorphism with fractional Dehn twist coefficient $\frac{k}{n}$ with $k\geq 5$.  Let $\Lambda$ and $\widehat{\Lambda}$ be a pair of disjoint Legendrian knots in $(M,\xi)$ constructed in Section~\ref{subsection: construction of Lambda}.

We give a heuristic argument for the exponential growth of strip Legendrian contact homology.  Assume that there is a fixed contact form $\alpha$ for $(M,\xi)$ that satisfies the conclusions of Theorems \ref{thm: once-punctured disks} and \ref{thm: strips}.  Let $\mathcal{G}(\tau)$ be the number of times the chord $\tau$ of $\Lambda$ or $\widehat\Lambda$ intersects the page $S\times\{{1\over 2}\}$. By Theorem~\ref{thm: strips} the differential for strip Legendrian contact homology counts strips in the mapping torus of $(S,h)$, i.e., if there is a holomorphic strip $\widetilde u$ between two chords $\tau,\tau'$, then $\mathcal{G}(\tau)=\mathcal{G}(\tau')$. Hence there exists a direct sum decomposition of  the strip Legendrian chain complex $LC_{st}(\alpha,\Lambda\to\widehat\Lambda)$ into
$$LC_{st}(\alpha,\Lambda\to\widehat\Lambda)=\oplus_{m\in \Z^{\geq 0}} LC_{st}^{\mathcal{G}=m} (\alpha,\Lambda\to\widehat\Lambda).$$

Now by the properties of pseudo-Anosov maps (cf.\ \cite{FLP} and \cite[Theorem 14.24]{FM}), the geometric intersection number between $L$ and $\psi^m(L)$ grows exponentially at a rate of $\lambda^m$, where $\lambda$ is the stretch factor for the pseudo-Anosov map.  Hence there exist numbers $a>0$ and $b$ such that:
\begin{equation} \label{eqn: exp growth}
\dim(LCH_{st}^{\mathcal{G}=m}(\Lambda \rightarrow \widehat{\Lambda}))\geq e^{am+b}.
\end{equation}
We also remark that we may replace ``$\mathcal{G}=m$'' by the condition ``$\leq m$'' on the action.

The difficulty is that, instead of a single contact form $\alpha$, we have constructed a sequence of contact forms for which we control orbits up to a certain action. The actual argument runs through direct limits as detailed in the next subsection.

\subsection{Growth of the number of Reeb chords}

As indicated in the previous subsection we need to use a direct limit argument.  The following is the key definition for taking direct limits:

\begin{defn} \label{defn: exhaustive}
Fix a contact form $\alpha_0$ for $(M,\xi)$.
A sequence $(\alpha_i=G_i\alpha_0, L_i)$, $i=1,2,\dots$, of pairs of contact forms and real numbers is {\em exhaustive} if there exists a constant $C>1$ such that
\begin{equation} \label{eqn: exhaustive condition}
L_{i+1}> CN_i N_{i+1}L_i
\end{equation}
where $N_i=\sup \{ G_i(x), \tfrac{1}{G_i(x)}~|~ x\in M\}$.
\end{defn}

Suppose that $(M,\xi)$ is a contact $3$-manifold satisfying the conditions of Theorem \ref{thm: main}.
Let us fix a contact form $\alpha_0=\alpha_{\epsilon_0 ,\epsilon'_0}$ with $\epsilon_0,\epsilon'_0$ sufficiently small.  Then we have the following analog of  \cite[Proposition 10.2]{CH2}:

\begin{lemma} \label{lemma: sequence of contact forms}
If $(M,\xi)$ is a contact $3$-manifold satisfying the conditions of Theorem \ref{thm: main}, then given a sequence $L_i$, $i=1,2,\dots$, going to $\infty$, there exists a sequence of contact forms $\alpha_{\epsilon_i,\epsilon'_i}$ with $\epsilon_i,\epsilon_i'\to 0$ such that:
\be
\item The strip contact homology chain groups $LC_{st}^{\leq L_i}(\alpha_{\epsilon_i,\epsilon'_i}, \Lambda\to \widehat\Lambda)$ are well-defined.
\item There exists an isotopy $(\phi^i_s)_{s\in[0,1]}$ of $M$ such that $G_i \alpha_0 = (\phi^i_1)^* \alpha_{\epsilon_i,\epsilon'_i}$
and $\frac{1}{4} \leq G_i \leq 4$.
\ee
\end{lemma}

We remark that the bound $\frac{1}{4^i}\leq G_i\leq 4^i$ that originally appears in \cite[Proposition 10.2]{CH2} can be improved to $\frac{1}{4} \leq G_i \leq 4$ by a slightly more careful use of \cite[Lemmas 10.4 and 10.5]{CH2} in the proof of \cite[Proposition 10.2]{CH2}.  Briefly, instead of comparing $\alpha_{\epsilon_0,\epsilon_0'}$ to $\alpha_{\epsilon_1,\epsilon_1'}$, $\alpha_{\epsilon_1,\epsilon_1'}$ to $\alpha_{\epsilon_2,\epsilon_2'}$, and so on, which gives the bound $\frac{1}{4^i}\leq G_i\leq 4^i$, we can directly compare $\alpha_{\epsilon_0,\epsilon_0'}$ to $\alpha_{\epsilon_i,\epsilon_i'}$, which gives the bound $\frac{1}{4}\leq G_i\leq 4$.

\begin{lemma}\label{lemma: criterion satisfied}
If $(M,\xi)$ is a contact $3$-manifold satisfying the conditions of Theorem \ref{thm: main}, then there exists an exhaustive sequence $(\alpha_i=G_i\alpha_0,L_i)$ that satisfies the conditions of Theorem~\ref{thm: criterion}.
\end{lemma}

\begin{proof}
(1) and (2). Pick $C>1$ and a sequence $L_i$ such that $L_{i+1}> 16 C  L_i$.  We then apply Lemma~\ref{lemma: sequence of contact forms} to obtain a sequence $\alpha_i=G_i\alpha_0$. Note that $$N_i:=\sup\{G_i(x),\tfrac{1}{G_i(x)}\}\leq 4$$ and $L_{i+1}> 16 C L_i$ implies that $L_{i+1}> C N_i N_{i+1} L_i$.  Hence we have an exhaustive sequence $(\alpha_i=G_i\alpha_0, L_i)$ to which Theorems~\ref{thm: once-punctured disks} and \ref{thm: strips} apply.  The strip Legendrian contact homology $LCH^{\leq L_i}_{st}(\alpha_i, \Lambda \rightarrow \widehat{\Lambda})$ is therefore well-defined and by Equation~\eqref{eqn: exp growth} there exist numbers $a>0$ and $b$ such that
\begin{equation}\label{eqrank}
\dim(LCH^{\leq  \ell}_{st}(\alpha_i,\Lambda \rightarrow \widehat{\Lambda})) \geq e^{a \ell+b}
\end{equation}
for each $i\in \mathbb{N}$ and $\ell\leq L_i$.  Hence $(\alpha_i=G_i\alpha_0,L_i)$ satisfies (1) and (2) of Theorem~\ref{thm: criterion} with ${\frak c}=4$.

(3).  By the argument of \cite[Section 10.3]{CH2}, if $i\leq j$, then the exhaustive condition (in particular Equation~\eqref{eqn: exhaustive condition}) and analogs of Theorems~\ref{thm: once-punctured disks} and \ref{thm: strips} for symplectic cobordisms imply that there exists an exact symplectic cobordism $\mathcal{W}^i_j$ and SFT-admissible exact Lagrangian cobordisms $\R\times\Lambda$ and $\R\times\widehat\Lambda$ that induce a map\footnote{The Lagrangians will be suppressed from the notation.}
\begin{equation}
\Psi_{\mathcal{W}^i_j}: LCH^{\leq L_i }_{st}(\alpha_i,\Lambda \rightarrow \widehat{\Lambda}) \to LCH^{\leq L_j}_{st}(\alpha_j,\Lambda \rightarrow \widehat{\Lambda}).
\end{equation}

In order to prove the injectivity statement, we consider the composition of cobordisms $\mathcal{W}^i_j$ and $\mathcal{W}^j_i$, whose induced maps are viewed as:
\begin{align}
\Psi_{\mathcal{W}^i_j}: & LCH^{\leq L_i/4^4 }_{st}(\alpha_i,\Lambda \rightarrow \widehat{\Lambda}) \to LCH^{\leq L_i/4^2}_{st}(\alpha_j,\Lambda \rightarrow \widehat{\Lambda});\\
\Psi_{\mathcal{W}^j_i}: & LCH^{\leq L_i/4^2}_{st}(\alpha_j,\Lambda \rightarrow \widehat{\Lambda}) \to LCH^{\leq L_i}_{st}(\alpha_i,\Lambda \rightarrow \widehat{\Lambda});
\end{align}
i.e., we have to restrict domains.
By the homotopy invariance and functoriality of cobordism maps, $ \Psi_{\mathcal{W}^i_i}= \Psi_{\mathcal{W}^j_i } \circ \Psi_{\mathcal{W}^i_j}$, which implies (3) with constant ${\frak d}=4^4$.
\end{proof}

\begin{prop} \label{growth}
If $(M,\xi)$ is a contact $3$-manifold satisfying the conditions of Theorem \ref{thm: criterion}, then there exist numbers $a>0$ and $b$
such that for every nondegenerate, $\Lambda\cup \widehat\Lambda$-nondegenerate contact form $\alpha=f_\alpha\alpha_0$ for $(M,\xi)$ we have
\begin{equation} \label{eqn: growth}
\#(\mathcal{T}^{\leq \mathsf{K}_\alpha \ell }_\alpha(\Lambda \rightarrow \widehat{\Lambda})) > e^{a\ell + b},
\end{equation}
where $\ell\in \R^+$ and  $\mathsf{K}_\alpha=\sup\{ f_\alpha(x), {1\over f_\alpha(x)}~|~ x\in M\}$.
\end{prop}

\begin{proof}
Let $(\alpha_i,L_i)$ be the exhaustive sequence satisfying Theorem~\ref{thm: criterion} and let $\mathsf{K}=\mathsf{K}_\alpha=\sup\{ f_\alpha(x), {1\over f_\alpha(x)}~|~ x\in M\}$.

For each pair of numbers $i$ and $j$ there exist an exact symplectic cobordism $\mathcal{W}^i(\alpha)$ from $\alpha_i $ to $\alpha$ (together with $\R\times (\Lambda\cup \widehat\Lambda)$) and an exact symplectic cobordism $\mathcal{W}_j(\alpha)$ from $\alpha$ to $\alpha_j$ (together with $\R\times (\Lambda\cup \widehat\Lambda)$) which induce maps
\begin{align}
\Psi_{\mathcal{W}^i(\alpha)}:& LCH^{\leq \ell}_{st}(\alpha_i, \Lambda\to \widehat\Lambda)\to LCH^{\leq \mathsf{K} {\frak c} \ell }_\epsilon (\alpha, \Lambda\to\widehat\Lambda),\\
\Psi_{\mathcal{W}_j(\alpha)}:& LCH^{\leq \mathsf{K} {\frak c} \ell}_\epsilon(\alpha,\Lambda\to\widehat\Lambda)\to LCH^{\leq \mathsf{K}^2{\frak c}^2 \ell}_{st} (\alpha_j,\Lambda\to\widehat\Lambda),
\end{align}
where $\ell< L_i$  and $\epsilon$ is the pullback of the trivial augmentation via the cobordism $\mathcal{W}_j(\alpha)$. Note that the trivial augmentation for $(\alpha_i, \Lambda\to\widehat\Lambda)$ is chain homotopic to the pullback of $\epsilon$ via $\mathcal{W}^i(\alpha)$, and hence induce isomorphic linearizations.

Since $L_j\geq \mathsf{K}^2 {\frak c}^2 L_i$ for $j\gg i$, in view of Equation~\eqref{eqn: exhaustive condition}, we can view $\Psi_{\mathcal{W}_j(\alpha)}$ instead as a map
$$\Psi_{\mathcal{W}_j(\alpha)}: LCH^{\leq \mathsf{K} {\frak c} \ell}_\epsilon(\alpha,\Lambda\to\widehat\Lambda)\to LCH^{\leq L_j}_{st} (\alpha_j,\Lambda\to\widehat\Lambda).$$

The gluing of $\mathcal{W}^i(\alpha)$ and $\mathcal{W}_j(\alpha)$ yields an exact symplectic cobordism $\mathcal{W}^i_j(\alpha)$ which is homotopic to $\mathcal{W}^i_j$ as exact symplectic cobordisms; we are still using SFT-admissible exact Lagrangians $\mathbb{R}\times (\Lambda\cup \widehat{\Lambda})$.

By the homotopy invariance and functoriality of cobordism maps,
\begin{equation}
\Psi_{\mathcal{W}^i_j}=\Psi_{\mathcal{W}^i_j(\alpha)} = \Psi_{\mathcal{W}_j(\alpha) } \circ \Psi_{\mathcal{W}^i(\alpha)}.
\end{equation}
This implies that:
\begin{equation}
\dim (LCH^{\leq \mathsf{K} {\frak c}\ell}_{\epsilon}(\alpha,\Lambda \rightarrow \widehat{\Lambda})) \geq \dim( \Psi_{\mathcal{W}^i_j}(LCH^{\leq \ell}_{st}( \alpha_i,\Lambda \rightarrow \widehat{\Lambda}))).
\end{equation}

On the other hand, by Conditions (2) and (3) of Theorem~\ref{thm: criterion},
\begin{equation}\label{eqestimate}
\dim(\Psi_{\mathcal{W}^i_j}(LCH^{\leq \ell}_{st}(\alpha_i,\Lambda \rightarrow \widehat{\Lambda}))) \geq e^{a \ell + b}
\end{equation}
for $\ell\leq L_i/{\frak d}$. Hence, for $\ell\leq L_i/{\frak d}$,
\begin{equation}
\#(\mathcal{T}^{\leq \mathsf{K}{\frak c}\ell}_\alpha(\Lambda \rightarrow \widehat{\Lambda})) \geq \dim( \Psi_{\mathcal{W}^i_j}(LCH^{\leq \ell}_{st}( \alpha_i,\Lambda \rightarrow \widehat{\Lambda}))) > e^{a\ell + b}.
\end{equation}
Finally, after renaming $a,b$, Equation~\eqref{eqn: growth} follows.
\end{proof}

\subsection{Positivity of topological entropy}

The goal of this subsection is to prove Theorem~\ref{thm: criterion}.

\s\n {\em Nondegenerate case.} First suppose that $\alpha$ is a nondegenerate contact form for $(M,\xi)$ and the conditions of Theorem~\ref{thm: criterion} are satisfied.

Fix $\delta>0$ small. By the neighborhood theorem for Legendrian submanifolds there exist a tubular neighborhood $\mathcal{U}_\delta(\widehat{\Lambda})$
of $\widehat{\Lambda}$ which is disjoint from $\Lambda$ and admits a contactomorphism
$$\Psi:(\mathcal{U}_\delta(\widehat{\Lambda}), \xi |_{\mathcal{U}_\delta(\widehat{\Lambda})})\stackrel\sim \longrightarrow (S^1 \times \mathbb{D}^2, \ker (\cos (\theta) dx - \sin (\theta) dy )),$$ for coordinates $\theta \in S^1$ and $(x,y) \in \mathbb{D}^2$, such that
\begin{itemize}
\item $\Psi(\widehat{\Lambda}) = S^1 \times \{(0,0)\}$;
\item all Legendrian curves of the form $\widehat{\Lambda}_{(x,y)}:=\Psi^{-1}(S^1 \times \{(x,y)\}) $ are $\delta$-close to $\widehat{\Lambda}$ in the $C^3$-topology.
\end{itemize}

The coordinates $(\theta,x,y)$ induce the structure of a Legendrian fibration $(\mathcal{U}_\delta(\widehat{\Lambda}), \xi |_{\mathcal{U}_\delta(\widehat{\Lambda})})$ with fibers $\widehat{\Lambda}_{(x,y)}$. Since $\widehat{\Lambda}_{(x,y)}$ is $\delta$-close to $\widehat{\Lambda}$, there exists a contactomorphism $\widetilde{\phi}_{(x,y)} : (M,\xi) \stackrel\sim\to (M,\xi)$  which is $\delta$-small in the $C^2$-topology, coincides with the identity outside $\mathcal{U}_\delta(\widehat{\Lambda})$, and satisfies $\widetilde{\phi}_{(x,y)} (\widehat{\Lambda}) = \widehat{\Lambda}_{(x,y)}$ for all $(x,y)$ in a sufficiently small ball around $(0,0)$.

The Reeb flows of $\widetilde{\phi}_{(x,y)}^* \alpha$ and $\alpha$ are conjugate by the map $\widetilde{\phi}_{(x,y)}$. Moreover this conjugation induces a bijection from $\mathcal{T}_{\widetilde{\phi}_{(x,y)}^*\alpha}(\Lambda \rightarrow \widehat{\Lambda}))$ to  $\mathcal{T}_\alpha(\Lambda \rightarrow \widehat{\Lambda}_{(x,y)})$, since it takes $\widetilde{\phi}_{(x,y)}^*\alpha$-Reeb chords from $\Lambda$ to $\widehat{\Lambda}$ to $\alpha$-Reeb chords $\Lambda$ to $\widehat{\Lambda}_{(x,y)}$. Because $\widetilde{\phi}_{(x,y)}^* \alpha$ is $\delta$-small in the $C^2$-topology we conclude that
\begin{equation}
\mathsf{K}':=\mathsf{K}_{\widetilde{\phi}_{(x,y)}^* \alpha} < \mathsf{K}_{\alpha} + \delta.
\end{equation}

By the argument of \cite[Lemma 3]{Al}, there exist a set $\mathcal{V}(\alpha) \subset \mathbb{D}^2$ of full measure with respect to the Lebesgue measure, such that for all $(x,y) \in \mathcal{V}(\alpha)$ the contact form $\alpha$ is $\Lambda\cup\widehat{\Lambda}_{(x,y)}$-nondegenerate. Applying Proposition \ref{growth} to $\widetilde{\phi}_{(x,y)}^* \alpha$, we obtain
\begin{equation}
\#(\mathcal{T}^{\leq \mathsf{K}' \ell}_\alpha(\Lambda \rightarrow \widehat{\Lambda}_{(x,y)})) =  \#(\mathcal{T}^{\leq\mathsf{K}' \ell}_ {\widetilde{\phi}_{(x,y)}^*\alpha}(\Lambda \rightarrow \widehat{\Lambda})) >e^{a\ell +b},
\end{equation}
or
\begin{equation}
\#(\mathcal{T}^{\leq \ell}_\alpha(\Lambda \rightarrow \widehat{\Lambda}_{(x,y)})) \geq  e^{\frac{a}{\mathsf{K}'} \ell +b}>  e^{\frac{a}{\mathsf{K}_{ \alpha} + \delta} \ell +b}
\end{equation}
for all $(x,y) \in \mathcal{V}(\alpha)$. Reasoning as in the proof of \cite[Theorem 1]{Al} we then conclude that
\begin{equation}
h_{top}(\phi_\alpha) \geq \tfrac{a}{\mathsf{K}_\alpha + \delta}.
\end{equation}
Since $\delta$ can be taken arbitrarily small we have
\begin{equation}
h_{top}(\phi_\alpha) \geq \tfrac{a}{\mathsf{K}_\alpha}.
\end{equation}

\s\n
{\em General case.}
Our reasoning so far implies that for any smooth contact form $\alpha$ for $(M,\xi)$ which is $\Lambda\cup\widehat{\Lambda}$-nondegenerate we have $h_{top}(\phi_\alpha) \geq \frac{a}{\mathsf{K}_\alpha}$. This conclusion can be extended to all smooth $\alpha$ for $(M,\xi)$ by observing that:
\begin{itemize}
\item the set of $C^\infty$-smooth contact forms for $(M,\xi)$ which are $\Lambda\cup\widehat{\Lambda}$-nondegenerate is dense in the set  of $C^\infty$-smooth contact forms for $(M,\xi)$;
\item $h_{top}(\phi_\alpha)$ depends continuously on the contact form $\alpha$ by a combination of the results of Newhouse \cite{New} and Katok \cite{Katok}; and
\item $\mathsf{K}_\alpha$ depends continuously on the contact form $\alpha$.
\end{itemize}

This completes the proof that every Reeb flow of $(M,\xi)$ satisfying the conditions of Theorem~\ref{thm: criterion} has positive topological entropy.

\begin{proof}[Proof of Theorem~\ref{thm: main}]
Follows from Lemma~\ref{lemma: criterion satisfied} and Theorem~\ref{thm: criterion}.
\end{proof}


\begin{thebibliography}{}

\bibitem[Ab]{Ab} C.\ Abbas, {\em Finite energy surfaces and the chord problem}, Duke Math.\ J.\ {\bf 96} (1999), 241--316.

\bibitem[Al1]{A}  M.\ R.\ R.\ Alves, {\em Cylindrical contact homology and topological entropy}, Geom.\ Topol.\ {\bf 20} (2016), 3519--3569.

\bibitem[Al2]{Al} M.\ R.\ R.\ Alves, {\em Legendrian contact homology and topological entropy}, preprint 2014. \texttt{ArXiv:1410.3381.}

\bibitem[Al3]{ReebAnosov} M.\ R.\ R.\ Alves, {\em Positive topological entropy for Reeb flows on $3$-dimensional Anosov contact manifolds}, J.\ Mod.\ Dyn.\ {\bf 10} (2016), 497--509.

\bibitem[B]{B} F.\ Bourgeois,  {\em A survey of contact homology}, in ``New perspectives and challenges in symplectic field theory'', 45--71, CRM Proc. Lecture Notes, 49, Amer. Math. Soc., RI, 2009.

\bibitem[BEE]{BEE} F.\ Bourgeois, T.\ Ekholm, and  Y.\ Eliashberg, {\em  Effect of Legendrian surgery}, Geom.\ Topol.\ {\bf 16} (2012), 301--389.

\bibitem[BEHWZ]{CPT} F.\ Bourgeois, Y.\ Eliashberg, H.\ Hofer, K.\ Wysocki, and E.\ Zehnder, {\em Compactness results in symplectic field theory}, Geom.\ Topol. {\bf 7} (2003), 799--888.

\bibitem[BH]{BH} E.\ Bao and K.\ Honda, {\em Semi-global Kuranishi charts and the definition of contact homology}, preprint 2015. \texttt{ArXiv:1512.00580.}

\bibitem[Bo]{Bo} R.\ Bowen, {\em Topological entropy and axiom A}, 1970 Global Analysis (Proc.\ Sympos.\ Pure Math., Vol. XIV, Berkeley, Calif., 1968) pp. 23--41 Amer.\ Math.\ Soc., Providence, R.I.

\bibitem[Ch]{Ch} Y.\ Chekanov, {\em Differential algebra of Legendrian links}, Invent.\ Math.\ {\bf 150} (2002), 441--483.

\bibitem[CH1]{CH1}
V.\ Colin and K.\ Honda, \textit{Stabilizing the monodromy of an open book decomposition}, Geom.\ Dedicata  {\bf 132}  (2008), 95--103.

\bibitem[CH2]{CH2}
V.\ Colin and K.\ Honda, \textit{Reeb vector fields and open book decompositions},  J.\ Eur.\ Math.\ Soc.\ (JEMS) {\bf 15} (2013), 443--507.

\bibitem[Ek]{Ek} T. \ Ekholm. \textit{Rational Symplectic Field Theory over $\Z_2$ for exact Lagrangian cobordisms.} J.\ Eur.\ Math.\ Soc.\ (JEMS) {\bf 10} (2008), 641--704.



\bibitem[EGH]{SFT}
Y.\ Eliashberg, A.\ Givental, and H.\ Hofer, {\em Introduction to symplectic field theory}, Geom.\ Funct.\ Anal.\ (2000), Special Volume, Part II, 560--673.

\bibitem[FLP]{FLP}
A.\ Fathi, F.\ Laudenbach, and V.\ Po\'enaru, {\em Thurston's work on surfaces}, Translated from the 1979 French original by Djun M.\ Kim and Dan Margalit. Mathematical Notes, 48. Princeton University Press, Princeton, NJ, 2012.

\bibitem[FM]{FM}
B.\ Farb and D.\ Margalit, {\em A primer on mapping class groups}, Princeton Mathematical Series, 49. Princeton University Press, Princeton, NJ, 2012.

\bibitem[FO3]{FO3}
K.\ Fukaya, Y.\ Oh, H.\ Ohta and K.\ Ono, \textit{Lagrangian intersection Floer theory:\ anomaly and obstruction, Part II,} AMS/IP Studies in Advanced Mathematics, 46.2. American Mathematical Society, Providence, RI; International Press, Somerville, MA, 2009.

\bibitem[FS1]{FS1}
U.\ Frauenfelder and F.\ Schlenk, {\em Volume growth in the component of the Dehn-Seidel twist}, Geom.\ Funct.\ Anal.\ {\bf 15} (2005), 809--838.

\bibitem[FS2]{FS2}
U.\ Frauenfelder and F.\ Schlenk, {\em Fiberwise volume growth via Lagrangian intersections},  J.\ Symplectic Geom.\ {\bf 4} (2006), 117--148.	

\bibitem[Gi]{Gi}
E.\ Giroux, {\em G\'eom\'etrie de contact : de la dimension trois vers les dimensions sup\'erieures}, Proceedings of the International Congress of Mathematicians, Vol. II (Beijing, 2002), 405--414, Higher Ed. Press, Beijing, 2002.

\bibitem[HL]{Halp} S. Halperin and J. M. Lemaire, {\em Notions of category in differential algebra}, Lecture notes in Mathematics, Vol. 1318, 138--154. Springer Verlag, Berlin, 1988.

\bibitem[H]{H} H.\ Hofer, {\em Pseudoholomorphic curves in symplectization with applications to the Weinstein conjecture in dimension three}, Invent.\ Math.\ {\bf 114} (1993), 515--563.

\bibitem[HWZ]{HWZ} H.\ Hofer, K.\ Wysocki, and E.\ Zehnder, {\em Properties of pseudoholomorphic curves in symplectisations I: Asymptotics}, Ann.\ Inst.\ H.\ Poincar\'e Anal.\ Non Lin\'eaire {\bf 13} (1996), 337--379.

\bibitem[HWZ2]{HWZ2}
H.\ Hofer, K.\ Wysocki and E.\ Zehnder,  \textit{A general Fredholm theory. I. A splicing-based differential geometry}, J.\ Eur.\ Math.\ Soc.\ (JEMS) {\bf 9} (2007), 841--876.

\bibitem[Kat1]{Katok} A.\ Katok, {\em Lyapunov exponents, entropy and periodic orbits for diffeomorphisms}, Inst.\ Hautes \'Etudes Sci.\ Publ.\ Math.\ {\bf 51} (1980), 131--173.

\bibitem[Kat2]{K2} A.\ Katok, {\em Entropy and closed geodesics}, Ergodic Theory Dynam.\ Systems {\bf 2} (1982), 339--365.

\bibitem[KH]{KH} A.\ Katok and B.\ Hasselblatt, {\em Introduction to the modern theory of dynamical systems. With a supplementary chapter by Katok and Leonardo Mendoza.} Encyclopedia of Mathematics and its Applications, 54. Cambridge University Press, Cambridge, 1995.

\bibitem[LS]{LS} Y.\ Lima and O.\ Sarig, {\em Symbolic dynamics for three dimensional flows with positive topological entropy}, preprint 2014. \texttt{ArXiv:1408.3427.}

\bibitem[MS]{MS} L.\ Macarini and F.\ Schlenk, {\em Positive topological entropy of Reeb flows on spherizations},  Math.\ Proc.\ Cambridge Philos.\ Soc.\ {\bf 151} (2011), 103--128.

\bibitem[New]{New} S.\ E.\ Newhouse, {\em Continuity properties of entropy}, Ann.\ of Math.\ (2) {\bf 129} (1989), 215--235.

\bibitem[Pa1]{Pa1} J.\ Pardon, {\em An algebraic approach to virtual fundamental cycles on moduli spaces of pseudo-holomorphic curves}, Geom.\ Topol.\ {\bf 20} (2016), 779--1034.

\bibitem[Pa2]{Pa2} J.\ Pardon, {\em Contact homology and virtual fundamental cycles}, preprint 2015. \texttt{ArXiv:1508.03873.}

\bibitem[Sar]{S} O.\ Sarig, {\em Symbolic dynamics for surface diffeomorphisms with positive entropy}, J.\ Amer.\ Math.\ Soc.\ {\bf 26} (2013), 341--426.

\bibitem[St]{St} {\itshape Stacks Project}, \texttt{Differential Graded Algebra.} Available at \href{http://stacks.math.columbia.edu/download/dga.pdf}{http://stacks.math.columbia.edu/download/dga.pdf}


\end{thebibliography}
\end{document}